\documentclass[a4paper, english, 11pt]{smfart} 

\usepackage{placeins} 

\usepackage{setspace}
\usepackage[british]{babel}

\usepackage[OT2,T1]{fontenc}
\usepackage{amssymb} 
\usepackage{mathrsfs} 

\usepackage{stmaryrd}
\usepackage{amsthm}

\usepackage{amsmath}
\usepackage[latin1]{inputenc}

\allowdisplaybreaks[1]

\usepackage{graphicx}
\usepackage{manfnt} 

 \usepackage{booktabs} 

\usepackage[margin=2.7cm]{geometry}

\usepackage[all]{xy}

\usepackage{math tools} 

\usepackage{enumitem}

\usepackage[x11names]{xcolor}

\usepackage[pagebackref=true, colorlinks=true, linkcolor=black, citecolor=black, urlcolor=black]{hyperref}
\usepackage{smfthm}

\usepackage[msc-links, alphabetic]{amsrefs}

\DeclareSymbolFont{cyrletters}{OT2}{wncyr}{m}{n}
\DeclareMathSymbol{\Sha}{\mathalpha}{cyrletters}{"58}

\newtheorem{theoA}{Theorem}

\newtheorem*{coro*}{Corollary}
\newtheorem*{conj*}{Conjecture}
\newtheorem*{lemm*}{Lemma}

\providecommand{\smalltwomat}[4]{\left(\begin{smallmatrix}#1&#2\\#3&#4\end{smallmatrix}\right)}

\theoremstyle{definition}

\theoremstyle{remark}
\newtheorem{remark*}{Remark}

\NumberTheoremsIn{section}

\numberwithin{equation}{section}
\numberwithin{table}{section}

\newcommand{\ot}{\otimes}

\newcommand{\ts}{\times}
\newcommand{\cd}{\cdot}

\newcommand{\beq}{\begin{equation}\begin{aligned}}
\newcommand{\eeq}{\end{aligned}\end{equation}}

\newcommand{\beqq}{\begin{equation*}\begin{aligned}}
\newcommand{\eeqq}{\end{aligned}\end{equation*}}

\newcommand{\lb}[1]{\label{#1}}

\newcommand{\nek}{Nekov{\'a}{\v{r}}}

\newcommand{\one}{\mathbf{1}}
\newcommand{\Q}{\mathbf{Q}}
\newcommand{\qqq}{\mathbf{q}}

\newcommand{\GL}{\mathrm{GL}}
\newcommand{\G}{\mathrm{G}}
\renewcommand{\H}{\mathrm{H}}

\newcommand{\X}{\mathscr{X}}

\newcommand{\R}{\mathbf{R}}
\newcommand{\Z}{\mathbf{Z}}

\newcommand{\frakm}{\mathfrak{m}}

\newcommand{\into}{\hookrightarrow}

\newcommand{\lan}{\langle}
\newcommand{\ran}{\rangle}

\newcommand{\al}{\alpha}

\newcommand{\lm}{\lambda}
\newcommand{\Lm}{\Lambda}
\newcommand{\sg}{\sigma}
\newcommand{\Sg}{\Sigma}

\newcommand{{\calG}}{\mathscr{G}}

\newcommand{\A}{\mathbf{A}}

\newcommand{\bks}{\backslash}

\newcommand{\eps}{\varepsilon}
\newcommand{\vphi}{\varphi}
\newcommand{\vth}{\vartheta}

\newcommand{\ord}{\mathrm{ord}} 

\newcommand{\vol}{\mathrm{vol}}
\newcommand{\Tr}{\mathrm{Tr}}

\newcommand{\Gal}{\mathrm{Gal}}

\newcommand{\Ch}{\mathrm{Ch}}

\newcommand{\Ker}{\mathrm{Ker}\,}

\newcommand{\Hom}{\mathrm{Hom}\,}
\newcommand{\End}{\mathrm{End}\,}

\newcommand{\llb}{\llbracket}
\newcommand{\rrb}{\rrbracket}

\newcommand{\Spec}{\mathrm{Spec}\,}

\renewcommand{\r}[1]{\mathrm{#1}}
\newcommand{\s}[1]{\mathscr{#1}}
\renewcommand{\sf}[1]{\mathsf{#1}}

\newcommand{\ol}[1]{\overline{#1}{}}

\newcommand{\ul}{\underline}
\renewcommand{\geq}{\geqslant}

\newcommand{\rB}{\r B}

\newcommand{\rG}{\r G}
\newcommand{\rH}{\r H}

\newcommand{\rM}{\r M}
\newcommand{\rN}{\r N}

\newcommand{\rP}{\r P}

\newcommand{\rT}{\r T}
\newcommand{\rU}{\r U}

\newcommand{\cyc}{\textstyle\circ}

\newcommand{\rc}{\r c}
\newcommand{\rd}{\,\r d}

\newcommand{\rt}{\r t}

\newcommand{\sA}{\s A}

\newcommand{\sH}{\s H}

\newcommand{\sK}{\s K}

\newcommand{\sO}{\s O}

\newcommand{\sR}{\s R}
\newcommand{\sS}{\s S}

\newcommand{\sV}{\s V}

\newcommand{\sX}{\s X}

\newcommand{\Qpb}{\ol{\Q}_{p}}

\newsavebox\tempbox
\let\svwidetilde\widetilde
\renewcommand\widetilde[1]{\sbox\tempbox{$#1$}\svwidetilde{\usebox{\tempbox}}}

   \def\XXint#1#2#3{{\setbox0=\hbox{$#1{#2#3}{\int}$}
        \vcenter{\hbox{$#2#3$}}\kern-.5\wd0}}

   \newcommand{\Herm}{\r{Herm}}

\title{Theta cycles 
and the Beilinson--Bloch--Kato conjectures}
\author{Daniel Disegni} 

\address{Department of Mathematics, Ben-Gurion University of the Negev, Be'er Sheva 84105, Israel}

\address{Aix-Marseille University, CNRS, I2M - Institut de Math\'ematiques de Marseille, campus de Luminy, 13288 Marseille, France}
\email{daniel.disegni@univ-amu.fr}

\begin{document}

\begin{abstract}
We introduce `canonical' classes in the Selmer groups of certain Galois representations with a conjugate-symplectic symmetry. They are images of  special cycles in unitary Shimura varieties, and defined uniquely up to a scalar. The construction is a slight refinement of one of Y. Liu, based on the conjectural modularity of Kudla's  theta series of special cycles. For $2$-dimensional representations, Theta cycles are (the Selmer images of) Heegner points. In general, they conjecturally exhibit an analogous strong  relation with the  Beilinson--Bloch--Kato conjectures  in rank~$1$, for which we gather the available  evidence. 
\end{abstract}

\thanks{Research supported by ISF grant 1963/20 and BSF grant 2018250. This work was partly written while the author was in residence at MSRI/SLMath (Berkeley, CA), supported by NSF grant DMS-1928930.}

\maketitle

\tableofcontents

 
 \section{Introduction}
 
 The purpose of this largely expository note is to introduce certain  Selmer classes of algebraic cycles,
 and discuss their relation to the Beilinson--Bloch--Kato (BBK) conjectures. These classes, called \emph{Theta cycles}, should play an analogous role to Heegner points on elliptic curves, in that the Bloch--Kato Selmer group $H^{1}_{f}(E, \rho)$ of a relevant Galois representation $\rho$ should be $1$-dimensional precisely when its Theta cycle is nonzero (cf. \cite{BST, Kim} and references therein for the case elliptic curves). Moreover, the  BBK conjectures, reviewed in  \S~\ref{sec 2}, predict that the $1$-dimensionality of the Selmer group is equivalent to the (complex or, for suitable primes, $p$-adic) $L$-function  of $\rho$ vanishing to\c order~$1$ at the center, and Theta cycles allow to approach this conjecture.

 \bigskip
 
The following theorem summarises the state of our knowledge on the topic.   Unexplained notions  or loose formulations  will be defined and made precise in the main body of the paper. 

We fix a rational prime $p$ and  denote by $\Q^{\cyc}\subset \Qpb$ the extension of $\Q$ generated by all roots of unity, and we fix an embedding $\iota^{\cyc}\colon \Q^{\cyc}\into \mathbf{C}$. We set $\Sg\coloneqq \{\iota\colon \Qpb\into \mathbf{C}\ | \ \iota_{|\Q^{\cyc}}=\iota^{\cyc}\}$. 

\begin{theoA} \lb{main} Let $E$ be a CM field with  Galois group $G_{E}$, and let 
$$\rho\colon G_{E}\to \GL_{n}(\Qpb)$$
 be an irreducible, geometric Galois representation of weight $-1$ and even dimension $n$. Suppose that $\rho$ is conjugate-symplectic, automorphic, and  has minimal regular Hodge--Tate weights. 
 
If $n\geq 4$, assume that the maximal totally real subfield $F$ of $E$ is not $\Q$, and that Hypothesis \ref{hyp coh} on the cohomology of  unitary Shimura varieties holds.

 \begin{enumerate}
 \item   Assume Hypothesis \ref{hyp mod} on the modularity of generating series of special cycles. The construction of \S~\ref{const} attaches to $\rho $ a pair $(\Lm_{\rho}, \Theta_{\rho})$, well-defined up to isomorphism, consisting of a $\Qpb$-line  $\Lm_{\rho}$ together with a $\Qpb$-linear map
 $$\Theta_{\rho}\colon \Lm_{\rho} \to H^{1}_{f}(E, \rho),$$
  whose image is spanned by classes of algebraic cycles.
  \item Suppose that $E$ and $\rho$ are `mildly ramified' and that $\rho$ is crystalline at $p$-adic places. 
\begin{enumerate}
\item Assume Hypothesis \ref{hyp mod}, as well as   Conjecture \ref{ass infty} on the injectivity of certain Abel--Jacobi maps, and that $p$ is unramified in  $E$.  For any     $\iota\in \Sg$, 
denote by  $ L_{\iota}(\rho,s)$ the   complex $L$-function of $\rho$ with respect to~$\iota$. Then\footnote{The order of vanishing of $L_{\iota}(\rho, s)$ at $s=0$ is conjecturally independent of $\iota$, cf. Conjecture \ref{bbk}.}
  $$\ord_{s=0} L_{\iota}(\rho,s) =1\ \Longrightarrow\  \Theta_{\rho}\neq 0.$$
\item Suppose that $E/F$ is totally split above $p$, that $p>n$, and that for every place $w\vert p$ of $E$, the representation $\rho_{w}$ is Panchishkin--ordinary. Denote by $\sX_{F}$ the $\Qpb$-scheme of continuous $p$-adic characters of $G_{F}$ that are unramified outside $p$, by $\frak m\subset\sO(\sX_{F})$ the ideal of functions vanishing at $\one$, and by $L_{p}(\rho)\in \sO(\sX_{F})$  the $p$-adic $L$-function of $\rho$. Then 
$$\ord_{\frak m} L_{p}(\rho) =1\ \Longrightarrow\ \text{Hypothesis \ref{hyp mod} holds and\ } \Theta_{\rho}\neq 0.$$
 \end{enumerate}
\item Assume Hypothesis \ref{hyp mod} and  that $\rho$ has `sufficiently large' image. Then 
$$\Theta_{\rho}\neq 0\  \Longrightarrow\ \dim_{\Qpb} H^{1}_{f}(E, \rho)=1.$$
 \end{enumerate}
\end{theoA}

Examples of representations $\rho$ satisfying the general assumptions of the theorem arise from  symmetric powers of elliptic curves: namely, if $A$ is a modular elliptic curve over $F$ with rational Tate module $V_{p}A$,  then by \cite{NT} one may consider the natural representation $\rho_{A, n}$ of $G_{E}$ on ${\rm Sym}^{n-1}V_{p}A_{E}(1-n/2)$ (see \cite[\S~1.4]{DL} for more details); in particular, for $n=2$ we obtain the representation $V_{p}A_{E}$ already studied (when $F=\Q$) by Gross--Zagier, Perrin-Riou and Kolyvagin in the 1980s.

\medskip

Part 1 of the theorem, which builds on constructions of Kudla and Y. Liu, is the main focus of this note; it is explained in \S~\ref{sec 4}, after reviewing the representation-theoretic preliminaries in \S~\ref{sec 3}. The construction is canonical up to a  representation-theoretic choice described in Remark \ref{can psi}. (However, there is  a `standard' choice, and part 3 of the theorem indicates that this ambiguity is quite innocuous.)

 In \S~\ref{sec 5}, we state a pair of formulas for the Bloch--Be\u\i linson and the \nek\ heights of Theta cycles, which are essentially reformulations of a breakthrough result of Li and Liu \cite{LL, LL2}, and of its $p$-adic analogue by Liu and the author \cite{DL}.   They imply the assertions of Part 2, and  take the  shape  
 $$\lan \Theta_{\rho} (\lm),
\Theta_{\rho^{*}(1)}^{}(\lm')\ran_{\star} 
= {c_{\star}\cd L_{\star}'(\rho, 0)}
\cd \zeta_{\star}(\lm, \lm'), 
$$
where `$\star$' stands for the relevant decorations,  $c_{\star}$ are constants, and $\zeta_{\star}$ are  canonical trivialisations of $\Lambda_{\rho}\ot \Lm_{\rho^{*}(1)}$.
 
Part 3 is the subject of  \cite{D-euler} (itself relying on forthcoming work of Jetchev--\nek--Skinner), on which we only give some brief remarks in \S~\ref{sec: euler}; in particular, we sketch the relevance of the perspective proposed here for the results obtained there.

\medskip

All the constructions and results should have analogues in the odd-dimensional case, in the symplectic case, and for more general Hodge--Tate types. We hope to return to some of these topics in future work.

  \subsection*{Acknowledgements} 
It will be clear to the reader that this note is  little more than an attempt to look from the Galois side, and the multiplicity-one side, at ideas of Kudla and Liu. I would like to thank Yifeng Liu for all I have learned from him during  our collaboration, and  Elad Zelingher for a remark that sparked it. I am also grateful to Yannan Qiu and Eitan Sayag for helpful conversations or correspondence, and to Chao Li and Yifeng Liu for many useful comments on a first draft. 
  
  This text is based on a talk given at the Second JNT Biennial Conference in Cetraro, Italy, in July 2022, and I would like  to thank the organisers for the opportunity to speak there. One of the participants reminded me of Tate's similarly named  `$\theta$-cycles' in the theory of mod-$p$ modular forms \cite[\S~7]{Jo}: besides the context, the capitalisation should also dispel any risk of confusion. Homonymous objects also occur in neuroscience, in connection with a pattern of brain activity typical of ``a drowsy state transitional from wake to sleep'' \cite[pp. 60-61]{McN}; I am grateful to the Cetraro audience for not indulging in this confusion either.

 \section{The conjecture of Be\u\i linson--Bloch--Kato--Perrin-Riou}
 \lb{sec 2}
Let $E$ be a number field with Galois group $G_{E}$,
 and let 
$$\rho\colon G_{E}\to \GL_{n}(\Qpb)$$
 be an irreducible Galois representation that is geometric in the sense of  \cite{FM95}, and pure of weight $-1$ at all finite places (in the sense of \cite[Definition A.11]{DL} -- where at non-$p$-adic places, we take the functor of \cite[(4.2.1)]{tate-nt} in place of the functor {${\rm WD}(\cdot )$ of \emph{loc. cit.}).

\begin{enonce*}[remark]{Example} The Galois representations attached to modular (eigencusp)forms are geometric and pure, see \cite{Saito1, Saito2}; the weight depends on the choice of normalisation, but if the modular form has even weight, a suitable cyclotomic twist of its Galois representation has weight~$-1$.
\end{enonce*}

\subsection{Chow and Selmer groups} A typical source of representations as above is the cohomology of algebraic varieties. 
In fact, define a \emph{motivation} of $\rho$ to be  an element of\footnote{Throughout this paper, if $R\to R'$ is a ring map that can be understood from the context, and $X$ is an $R$-scheme or an $R$-module, we write $X_{R'}\coloneqq X\ot_{R}R'$.}
$${\r{Mot}}_{\rho}\coloneqq \varinjlim_{(X, k)} \r{Mot}_{\rho}(X, k),
\qquad \text{where }\quad 
\r{Mot}_{\rho}(X, k)\coloneqq 
 \Hom_{\Qpb[G_{E}]} (H_{\text{\'et}}^{2k-1}(X_{\ol{E}}, \Qpb(k)), \rho),$$
and the limit runs over all pairs consisting of a smooth proper variety $X_{/E}$ and an integer $k\geq 1$ (this is a directed system by K\"unneth's formula). 
We refer to
 $(X, k)$ as a source of $f \in {\r{Mot}}_{\rho}$ if $f$ is in the image of $\r{Mot}_{\rho}(X, k)$. We say that $\rho $ is \emph{motivic} if ${\r{Mot}}_{\rho}$ is nonzero. 
 According to the conjecture of Fontaine--Mazur,  every geometric irreducible Galois representation is motivic.

To a  representation $\rho$ as above is attached its Bloch--Kato  \cite{BK} Selmer group $H^{1}_{f}(E, \rho)$.\footnote{N.B.: the subscript $f$ has nothing to do with names of objects elsewhere in this text. Galois cohomology and Selmer groups are usually defined for representations with coefficients in finite extensions of $\Q_{p}$. However, it is well-known that we can write $\rho=\rho_{0}\ot_{L}\Qpb$ for some finite extension $L\subset \Qpb$ of $\Q_{p}$ and some representation $\rho_{0}\colon G_{E}\to \GL_{n}(L)$ (and similarly for the other representations considered in this paper). Then we define $H^{1}_{f}(E, \rho)\coloneqq  H^{1}_{f}(E, \rho_{0})\ot_{L}\Qpb$.} To a
variety $X_{/E}$ as above is attached  its Chow group  $\Ch^{k}(X)$ of  codimension-$k$ algebraic cycles on $X$ up to rational equivalence (with coefficients in $\Q$). A central object of arithmetic interest is its subgroup $\Ch^{k}(X)^{0}_{\Qpb}\coloneqq  \Ker [\Ch^{k}(X)_{}\to H^{2k}_{\text{\'et}}(X_{\ol{E}}, \Qpb(k))]$ (where the map is the cycle class). It is endowed with an Abel--Jacobi map 
$$\r{AJ}\colon \Ch^{k}(X)^{0}_{\Qpb} \to  H^{1}(E,  H_{\text{\'et}}^{2k-1}(X_{\ol{E}}, \Qpb(k)))$$
 (see \cite[\S~5.1]{nek-height}) whose image is conjectured to land in  $H^{1}_{f}(E,  H_{\text{\'et}}^{2k-1}(X_{\ol{E}}, \Qpb(k)))$.
We can define an analogue of the image of $\r{AJ}$ for the representation $\rho$ 
by
$$H^{1}_{f}(E, \rho)^{\r{mot}} \coloneqq  \sum_{f'\in \r{Mot}_{\rho}} f'_{*}{\r{AJ}}(\Ch^{k}(X)_{\Qpb}^{0}) \cap H^{1}_{f}(E, \rho)
\subset H^{1}_{f}(E, \rho),$$
where we have denoted by $(X, k)$  any source of the motivation  $f'$. 
By an evocative abuse of nomenclature, we refer to elements of $H^{1}_{f}(E, \rho)^{\r{mot}}$ as cycles. 
\begin{rema} If $\rho=  H_{\text{\'et}}^{2k_{0}-1}(X_{0,\ol{E}}, \Qpb(k_{0}))$ for a variety $X_{0}$ and an integer $k_{0}$, then we expect that $H^{1}_{f}(E, \rho)^{\r{mot}}={\r{AJ}}(\Ch^{k_{0}}(X_{0})_{\Qpb}^{0})$. This equality is implied by the Tate conjecture \cite[Conjecture 1]{Tat65} for $X\times X_{0}$.
\end{rema}

\subsection{The conjecture} 
We say  that $\rho$ is  (Panchishkin-) \emph{ordinary} (see \cite[\S~6.7]{nek-height}, \cite[\S~2.3.1]{PR-htIw} for more details) if for each place $w\vert p$, there is a (necessarily unique) exact sequence of De Rham $G_{E_{w}}$-representations $0\to \rho_{w}^{+} \to \rho_{|G_{E_{w}}}\to \rho_{w}^{-}\to 0$, such that $\r{Fil}^{0}\mathbf{D}_{\r{dR}}(\rho_{w}^{+}) = \mathbf{D}_{\r{dR}}(\rho_{w}^{-})/{\r{Fil}}^{0}=0$.  For any subfield $F\subset E$, let  
$$\X_{F}\coloneqq \Spec \Z_{p}\llb \Gal(F_{\infty}/F)\rrb\ot_{\Z_{p}} \Qpb, $$ 
where $F_{\infty}/F$ is the abelian extension with $\Gal(F_{\infty}/F)$ isomorphic (via class field theory) to the maximal $\Z_{p}$-free quotient of $F^{\ts}\bks \A_{F}^{\ts} / \widehat{\sO_{F}}^{p, \ts}$. 

One can conjecturally attach to $\rho$ entire $L$-functions $$L_{\iota}(\rho, s)$$ for ${\iota\colon L\into\mathbf{C}}$ and, (at least) if $\rho$ is ordinary,  a $p$-adic $L$-function $$L_{p}(\rho) \in \sO(\sX_{F})$$ interpolating suitable modifications of the $L$-values $L_{\iota}(\rho\ot\chi_{|G_{E}}, 0)$
 for finite-order characters $\chi\in  \sX_{F}$ (see \cite{PRbook}, at least when  taking $F=\Q$). 
 
 Denote by $\frakm=\frakm_{F}\subset \sO(\sX_{F})$ the maximal ideal of functions vanishing at the character $\one$ of $\Gal(F_{\infty}/F)$, and by $\ord_{\frakm}$ the corresponding valuation. The integer ${\ord}_{\frakm} L_{p}(\rho)$ is conjecturally independent of the choice of $F$.
 
\begin{conj}[{Be\u\i linson}, Bloch--Kato, Perrin-Riou {\cite{bei, BK, PRbook}}]  \lb{bbk} Let $\rho\colon G_{E}\to \GL_{n}(\Qpb)$ be an irreducible geometric representation of weight $-1$. Let $r\geq 0$ be an integer.   The following conditions are equivalent:
\begin{enumerate}
\item[$\text{(a)}_{\infty}$] for any $\iota\colon \Qpb\into \mathbf{C}$, we have $${\ord}_{s=0} L_{\iota}(\rho,s) = r;$$
\item[$\text{(b)}_{\phantom\infty}$] $\dim_{\Qpb} H^{1}_{f}(E, \rho)^{\r{mot}} =\dim_{\Qpb} H^{1}_{f}(E, \rho) = r$.
\end{enumerate}
If moreover  $\rho$ is ordinary and  $\rho_{w}^{+, *}(1)^{G_{E_{w}}}=0$ for every $w|p$, then the above conditions are equivalent to
\begin{enumerate}
\item[$\text{(a)}_{p\ }$] 
 ${\ord}_{\frakm} L_{p}(\rho) = r;$
\end{enumerate}
%
\end{conj}

\begin{rema} The first  equality in (b) generalises the conjectural finiteness of the $p^{\infty}$-torsion in the Tate--Shafarevich group of an elliptic curve. The extra condition in $\text{(a)}_{p\ }$ serves to avoid  the phenomenon of exceptional zeros, cf. \cite{Ben}.
\end{rema}

In the following pages, under some restrictions on $\rho$ we will define elements in $H^{1}_{f}(E, \rho)^{\rm mot}$ whose nonvanishing is conjecturally equivalent to the conditions of Conjecture \ref{bbk} with $r=1$. The construction will be automorphic; in the next section, we give the representation-theoretic background.

\section{Descent and  theta correspondence} 
\lb{sec 3}
Suppose for the rest of this paper that $E$ is a CM field with totally real subfield $F$.   We denote by ${\rm c}\in \Gal(E/F)$  the complex conjugation, and by $\eta\colon F^{\ts}\bks\A^{\ts}\to\{\pm 1\}$   the quadratic character attached to $E/F$. 
\subsection{$p$-adic automorphic representations} 
 We denote by $\A$ the ad\`eles of $F$; if $S$ is a finite set of places of $F$, we denote by $\A^{S}$  the ad\`eles of $F$ away from $S$. If $\G$ is a group over $F$ and $v$ is a place of $F$, we write $G_{v}\coloneqq \G(F_{v})$; if $S$ a finite set of places of $F$, we write $G_{S}\coloneqq \prod_{v\in S}G(F_{S})$. (For notational purposes, we will identify a place of $\Q$ with the set of places of $F$ above it.) We denote by  $\psi\colon F\bks \A\to \mathbf{C}^{\ts}$ the standard additive character with $\psi_{\infty}(x)=e^{2\pi i \Tr_{F_{\infty}/\R}x}$,
  and we set $\psi_{E}\coloneqq \psi\circ \Tr_{E/F}$.  We view $\psi_{|\A^{\infty}}$ as valued in $\Q^{\cyc}$ via the embedding $\iota^{\cyc}$.

\subsubsection{Unitary groups} 
Fix a positive integer $n$.  For a  place $v$ of $F$, we denote by  $\sV_{v}$ be the set of isomorphism classes of  (nondegenerate) $E_{v}/F_{v}$-hermitian spaces of dimension $n$; this consists of one element if $v$ splits in $E$, of two elements if $v$ is finite nonsplit, and of $n+1$ elements if $v$ is real. We denote by $\sV^{+}$ the set of isomorphism classes of $E/F$-hermitian spaces of dimension $n$ that are positive definite at all archimedean place, and by $\sV^{-}$ the set of   isomorphism classes of $E/F$-hermitian spaces of dimension $n$ that are positive definite at all archimedean place but one,  at which the signature is $(n-1, 1)$. We denote by $\sV^{\circ}$ the set of 
isomorphism classes of $\A_{E}/\A$-hermitian spaces of dimension $n$ such that  for all but finitely many places $v$, the Hasse--Witt invariant $\epsilon(V_{v})\coloneqq \eta_{v}( (-1)^{n\choose 2}\det V_{v})=+1$, and that $V_{v}$ is positive definite at all archimedean places. We put $\epsilon (V)\coloneqq \prod_{v}\epsilon(V_{v})$, and write $\sV^{\circ, \epsilon}\subset \sV^{\circ}$ for the set of spaces with $\epsilon (V)=\epsilon\in \{\pm\}$.

We have a natural identification $\sV^{\circ,+}=\sV^{+}$. We will mostly be interested in $ \sV^{\circ, -}$, which  we refer to as the set of   \emph{incoherent} $E/F$-hermitian spaces, cf. \cite{gross-incoh}.  If $V\in \sV^{\circ, -}$, then for every archimedean place $v$ of $F$, there exists a unique $V(v)\in \sV^{-}$ over $F$ such that $V(v)_{w}\cong V_{w}$ if $w\neq v$. 

For $V\in \sV$, let $\rH_{V}=\rU(V)$; if $V\in \sV^{\circ}$  with $\epsilon (V)=-1$,  we still use the notation $\H_{V}(\A^{S})\coloneqq \prod_{v\notin S}H_{V_{v}}$, $H_{V_{v}}\coloneqq \rU(V_{v})(F_{v})$, and we refer to (the symbol)
$$\H_{V}$$ as an \emph{incoherent} unitary group.

Suppose from now on that $n=2r$ is even. We define the quasisplit unitary group over $F$ 
$$\G=\rU(W),$$ 
where $W=E^{n}$ equipped with  the skew-hermitian form $\smalltwomat{}{ 1_{r}}{-1_{r}}{}$ (here $1_{r}$ is the identity matrix of size $r$).

\begin{defi}\lb{defR1}
 \begin{enumerate}
\item
A \emph{relevant} complex automorphic representation $\Pi$ of $\GL_{n}(\A_{E})$ is an irreducible cuspidal automorphic representation  satisfying: 
\begin{enumerate}
\item[(i)]  $\Pi\circ {\rm c}\cong \Pi^{\vee}$;
\item[(ii)] for every archimedean place $w$ of $E$, the representation $\Pi_{w}$ is induced from the character ${\rm arg}^{n -1}\ot {\rm arg}^{n -3}\ot\ldots\ot{\rm arg}^{1-n}$ of the torus $(\mathbf{C}^{\ts})^{n}=(E_{w}^{\ts})^{n}\subset \GL_{n}(E_{w})$; here ${\rm arg}(z)\coloneqq  z/|z|$.
\end{enumerate}
 \item A \emph{possibly relevant} complex automorphic representation $\pi$ of $\G(\A)$ is an irreducible cuspidal automorphic representation  such that for every archimedean place $v$ of $F$, the representation 
 $\pi_v$ is the holomorphic discrete series representation of Harish-Chandra parameter $\{\tfrac{n-1}{2},\tfrac{n-3}{2},\dots,\tfrac{3-n}{2},\tfrac{1-n}{2}\}$. We say that $\pi$ is \emph{relevant}  if it is possibly relevant and \emph{stable} as defined at the beginning of \S~\ref{sec:desc} below. 
 \item Let $V\in \sV^{\circ, -}$ and let $v$ be an archimedean place of $F$. A \emph{possibly relevant} complex cuspidal automorphic representation $\sg$ of $\H_{V(v)}(\A)$ is an irreducible cuspidal automorphic representation such that  $\sg_{v}$  is one of the $n$  discrete series representation of $H_{V(v)_{v}}=U(n-1,1)$ of Harish-Chandra parameter $\{\tfrac{n-1}{2},\tfrac{n-3}{2},\dots,\tfrac{3-n}{2},\tfrac{1-n}{2}\}$, and for every other archimedean place $v'\neq v$ of $F$, we have $\sg_{v'}=\one$ (as a representation of $H_{V(v)_{v'}}=U(n)$).  We say that $\sg$ is \emph{relevant} if it is possibly relevant and stable.
\end{enumerate}
\end{defi}

\begin{defi}\lb{defR2}
 \begin{enumerate}
\item
A \emph{relevant} $p$-adic automorphic representation $\Pi$ of $\GL_{n}(\A_{E})$ is a representation of $\GL_{n}(\A_{E}^{\infty})$ on a $\Qpb$-vector space,  such that for every $\iota\colon \Qpb\into\mathbf{C}$, the representation $\iota\Pi$ is the finite component of a (unique up to isomorphism) relevant complex automorphic representation $\Pi^{\iota}$.
 \item A \emph{possibly relevant}, respectively \emph{relevant} $p$-adic  automorphic representation $\pi$ of $\G(\A)$ is representation of $\G(\A^{\infty})$  on a $\Qpb$-vector space,  
 such that  for every $\iota\colon \Qpb\into\mathbf{C}$, the representation $\iota\pi$ is the finite component of a (unique up to isomorphism)  possibly  relevant, respectively relevant,  complex automorphic representation $\pi^{\iota}$ of $\G(\A)$.
 \item  Let $V\in \sV^{\circ, -}$.
  A \emph{possibly relevant}, respectively \emph{relevant}, $p$-adic  automorphic representation $\sg$ of $\H_{V}(\A)$ is representation of $\H_{V}(\A^{\infty})$  on a $\Qpb$-vector space,  
 such that  for every $\iota\colon \Qpb\into\mathbf{C}$ and every archimedean place $v$ of $F$, the representation $\iota\sg$ is the finite component of a (unique up to isomorphism) possibly relevant, respectively relevant, complex automorphic representation $\sg^{\iota, (v)}$ of $\H_{V(v)}(\A)$.
%
\end{enumerate}
\end{defi}

\subsection{Automorphic descent}   \lb{sec:desc}
For a place $v$ of $F$, we denote by ${\rm BC}_{v}$ the base-change map from $L$-packets  of tempered representations of $G_{v}$ to tempered representations of $\GL_{n}(E_{v})$, which is injective by \cite[Lemma 2.2.1]{Mok}. We denote by ${\rm BC}_{\G}$ and ${\rm BC}_{\H_{V}}$ the base-change maps from automorphic representations of the unitary groups $\G(\A)$ or $\H_{V}(\A)$  to automorphic representations of $\GL_{n}(\A_{E})$, respectively; we simply write ${\rm BC} $ when there is no risk of confusion. We say that a cuspidal automorphic representation of a unitary group is \emph{stable} if its base-change is still cuspidal.

\begin{rema} \lb{BC p}
 We have the following properties of the base-change maps.
\begin{enumerate}
\item[(a)]
 By  \cite[Proposition C.3.1]{LTXZZ}, if $\Pi$ is a relevant representation of $\GL_{n}(\A_{E})$, then: the preimage of $\Pi$ under ${\rm BC}_{\H_{V}}$ 
 consists of relevant   representations of $\H_{V}(\A)$; the preimage of $\Pi$  under ${\rm BC}_{\G}$   contains  a relevant representation of $\G(\A)$.
 \item[(b)]
 If $v$ is a finite place, the base-change maps may be defined for representations with coefficients over $\Qpb$, compatibly with any extensions of scalars $\iota\colon \Qpb\into\mathbf{C}$. 
 \item[(c)] As a consequence of (a) and (b), ${\rm BC}$ extends to a map from  relevant \emph{$p$-adic} automorphic representations of $\G(\A)$ and $\H_{V}(\A)$ to  relevant {$p$-adic} automorphic representations of $\GL_{n}(\A_{E})$.
\end{enumerate}
\end{rema}

 \subsubsection{Descent to a quasisplit unitary group}
We fix the auxiliary choice of a Borel subgroup $\rB\subset\rG$ with torus $\rT$ and  unipotent radical $\rN$, and (the $\rT$-orbit of) a generic  linear homomorphism $\Psi\colon \rN(F)\bks \rN(\A)\to \mathbf{C}^{\ts}$; we call this choice $(\rN, \Psi)$ a \emph{Whittaker datum}. A relevant complex or $p$-adic  automorphic representation $\pi$ of $\G(\A)$ is called \emph{$ \Psi$-generic} if it for every finite place,   $\pi_{v}$  is $\Psi_{v}$-generic in the sense that it has a non-vanishing $(N_{v},  \Psi_{|N_{v}})$-Whittaker functional .

\begin{prop} \lb{descent}
Let $\Pi$ be a  relevant $p$-adic automorphic representation of $\GL_{n}(\A_{E})$. Then there exists a relevant $p$-adic automorphic representation $\pi$ of $\G(\A)$, unique up to isomorphism, which is $\Psi$-generic and satisfies 
${\rm BC}(\pi)=\Pi$. 
\end{prop}

\begin{proof}
By \cite{GRS} and \cite{Mor}, for each $\iota$ there exists a   relevant cuspidal automorphic representation  $\pi^{\iota}$ of $\G(\A)$ that is $\Psi$-generic and satisfies ${\rm BC}(\pi^{\iota})=\Pi^{\iota}$.
 By \cite{Varma, Ato}, for each finite place $v$, each local $L$-packet of $G_{v}$ contains a unique $\Psi$-generic representation, which (together with the injectivity of ${\rm BC}_{v}$) implies that $\pi^{\iota}$ is unique up to isomorphism. Then by Remark \ref{BC p} (b), the collection $(\pi^{\iota})$ arises from a well-defined relevant $p$-adic automorphic representation $\pi$ of $\G(\A)$.
\end{proof}

\begin{rema}\lb{can psi} 
Our construction of Theta cycles will be based on the choice of a relevant  representation  $\pi$ with ${\rm BC}(\pi)=\Pi$, which is not unique. For definiteness, we may pick a Whittaker datum $\Psi$ (for which, as explained in \cite[\S~0.2.2, \S~1.6.1]{KMSW}, there is a standard choice), and take $\pi$ to be the $\Psi$-generic representation given by Proposition \ref{descent}.
\end{rema}




\subsection{Theta correspondence}  Let $\pi$ be a relevant $p$-adic representation of $\G$ with ${\rm BC}(\pi)=\Pi$. We will need to further transfer $\pi$ to a representation of unitary groups $\H_{V}$ for $V\in \sV^{\circ , - }$. 

\subsubsection{Local correspondence and duality} We first review the local theory.
Let $v$ be a  finite place of $F$, and let $C$ be either $\Qpb$ or $\mathbf{C}$.
  For $V_{v}\in \sV_{v}$, let $\omega_{V_{v}}=\omega_{V_{v}, \psi_{v}}$ be the Weil representation of $H_{V_{v}}\times G_{v}$  (with respect to the 
  character $\psi_{v}$) over $C$, a model of which is recalled in \S~\ref{models} below. 
  
  Whenever $\Box$ is some smooth admissible  representation of a group $G^{?}$, we denote by $\Box^\vee$ the contragredient, and  by  $(\, ,\ )_{\Box}$ the natural pairing on $\Box\times \Box^{\vee}$.

  The first part of the following result (for nonsplit finite places) is known as \emph{theta dichotomy}.
 
\begin{prop} \lb{dicho}
Let $\pi_{v}$ be an tempered irreducible admissible representation of $G_{v}$ over $C=\Qpb$ or $C=\mathbf{C}$.
\begin{enumerate}
\item
There exists a unique $V_{v}\in \sV_{v}$ such that 
$$\sg_{v}^{\vee}\coloneqq (\pi_{v}^{\vee} \ot \omega_{V_{v}})_{G_{v}}\neq 0.$$
\item 
The representation $\sg_{v}^{\vee}$ is tempered and irreducible.  Its contragredient $\sg_{v}$ satisfies ${\rm BC}(\sg_{v})={\rm BC}(\pi_{v})$, and the space
$$\Hom_{H_{V_{v}} \ts G_{v}}(\pi^{\vee}_{v} \ot \omega_{V_{v}} \ot \sg_{v}, C)$$
is $1$-dimensional over $C$.
\item 
The representation $(\pi_{v} \ot \omega^{\vee}_{V_{v}})_{G_{v}}$ is canonically identified with  $\sg_{v}$.
\item Denote by $\vartheta$ each of the projection maps $\pi_{v}^{\vee} \ot \omega_{V_{v}}\to \sg_{v}^{\vee}$, $\pi_{v} \ot \omega^{\vee}_{V_{v}}\to \sg_{v}$.
Then the map
$$\zeta_{v}( \vphi, \phi, f;  \vphi', \phi', f') \coloneqq  ( \vartheta(\vphi, \phi), f)_{\sg_{v}^{\vee}}\cdot  ( \vartheta(\vphi', \phi'), f')_{\sg_{v}^{}}$$
defines a canonical generator
$$\zeta_{v}\in 
\Hom_{G_{v}\ts H_{V_{v}}} 
(\pi^{\vee}_{v} \ot \omega_{V_{v}} \ot\sg_{v}, C)
\ot_{C}
\Hom_{G_{v}\ts H_{V_{v}}} 
( \pi^{}_{v} \ot \omega^{\vee}_{V_{v}} \ot (\sg_{v}^{\vee}, C),$$
with the property that if $\pi_{v}$ and $\sg_{v}$ are unramified and $ \vphi, \phi, f,  \vphi', \phi', f'$ are spherical vectors, then 
 $$\zeta_{v} (\vphi, \phi, f;  \vphi', \phi',f' ) = (\vphi, \vphi')_{\pi_{v}^{\vee}} \cd(\phi, \phi')_{\omega_{v}^{\vee}}  (f, f')_{\sg_{v}}.$$
\end{enumerate}
\end{prop}
\begin{proof}
We drop all subscripts $v$. We start by recalling the first two statements. Consider first the case that $v$ is finite and $E$ is a field. 
Then  $\sg_{V}^{\vee}=(\pi^{\vee} \ot \omega_{V})_{G}$ is the (a priori, `big') theta lift of $\pi^{\vee}$ as defined in \cite[(2.1.5.1)]{harris}. 
By the local theta dichotomy proved in Theorem 2.1.7 (iv) \emph{ibid.} and \cite[Theorem 3.10]{GG11}, there is exactly one $V\in \sV$ such that $\sg_{V}^{\vee}$ is nonzero; we fix this $V$ and drop it from then notation. Then the other properties of $\sg\coloneqq (\sg^{\vee})^{\vee}$ are consequences of \cite[Theorem 4.1]{GI16} (which collects results from \cite{Wald-th, GT, GS12, GI14}).  For the case $E=F\oplus F$, see \cite{Min08}.

We now turn to the other two statements. For a character $\chi\colon F^{\ts}\to C^{\ts}$, let 
\beq \lb{bnchi} b_{n}(\chi)\coloneqq  \prod_{i=1}^{n} L(i, \chi\eta^{i-1}).\eeq
If $C=\mathbf{C}$,
then we have a  canonical element
$$\breve\zeta \in \Hom_{G}( \pi^{\vee}\ot \omega_{V}, \mathbf{C})\ot_{\mathbf{C}}
\Hom_{ G}( \pi^{} \ot \omega^{\vee}_{V}, \mathbf{C})$$
given by 
\beq \lb{zbr}
\breve\zeta ( \vphi, \phi;  \vphi', \phi')\coloneqq  
 {b_{n}(\one)\over L(1/2, \Pi)} 
\int_{G}     (g\vphi, \vphi')_{\pi^{\vee}} \cdot 
(\omega(g)\phi, \phi')_{\omega} \, dg,\eeq
where $dg$ is the measure of \cite[\S~2.1 (G7)]{DL},  $\Pi\coloneqq {\rm BC}(\pi)$. It is a generator by  \cite[\S~6]{HKS}, where the regularisation of the integral is also taken care of. (For the well-known comparison between the definition in \emph{loc. cit.}  and the one given here, see \cite[Lemma 3.1.2]{Sak}.)
 When $\pi$ (hence $\sg$) are unramified and all the vectors are spherical, by   \cite[Proposition 7.1, (7.2)]{Yam14}
we  have 
\beq \lb{unr breve} \breve\zeta (\vphi, \phi;  \vphi', \phi') =  (\vphi, \vphi')_{\pi^{\vee}} \cd(\phi, \phi')_{\omega^{\vee}}.
\eeq
If $C=\Qpb$, then for any $\iota\in \Sg$ we have a tetralinear form  $\breve\zeta^{\iota}$ as above, and by \cite[Lemma 3.30]{DL}, there is a $\breve\zeta \in \Hom_{G}( \pi^{\vee}\ot \omega_{V}, \Qpb)\ot_{\Qpb}
\Hom_{ G}( \pi^{} \ot \omega^{\vee}_{V}, \Qpb)$ such that $\zeta\ot_{\Qpb, \iota}1=\zeta^{\iota}$ for every $\iota\in \Sg$. 

Now, we may view $\breve\zeta$ as a map 
\beq \lb{z breve}
\breve{\zeta}\colon (\pi^{\vee} \ot \omega_{V})_{G}\ot ( \pi^{} \ot \omega^{\vee}_{V})_{G}\to C\eeq
that is, by inspection, invariant under the diagonal action of $H$ on both factors. It follows that $\breve{\zeta}$ gives the duality  of our third statement. The fourth statement then follows from the definitions and \eqref{unr breve}.
\end{proof}

\begin{rema} A more symmetrically  defined  exalinear form would be 
$$
 (\vphi, \phi, f;  \vphi', \phi', f')\mapsto \int_{H_{V}}\int_{G}   (g\vphi, \vphi')_{\pi^{\vee}} \cdot 
(\omega(h,g)\phi, \phi')_{\omega} \cd    (hf, f')_{\sg} \, dgdh,$$
where the integral in $dg$ is regularised as remarked after \eqref{zbr}.
If $\sg$ is  a discrete series, the integral in $dh$
 converges and its value  equals that of $\zeta_{v}$, times the formal degree of $\sg$ -- for which  \cite{BP-Planch} gives a formula in terms of adjoint gamma factors. In general, regularising the integral in $dh$ amounts to regularising  the inner product of two matrix coefficients of $\sg$.  A regularisation has been proposed by Qiu \cite{Qiu, Qiu2}; however the definition of the resulting generalised formal degree is partly conjectural, and no precise  (even conjectural) formula for it appears in the literature. 
\end{rema}

\subsubsection{Global correspondence} We have  the following global variant of Proposition \ref{dicho}.
\begin{prop}\lb{th incoh}
Let $\Pi$ be a relevant $p$-adic automorphic representation of $\GL_{n}(\A_{E})$, and set $\epsilon= \epsilon(1/2, \Pi)$.  Let $\sR_{\Pi, \G}$ be the set of isomorphism classes of relevant automorphic  representations $\pi$ of $\G(\A)$ with ${\rm BC}(\pi)=\Pi$, and let $\sR_{\Pi, \H}$ be the set of pairs $(V, \sg)$, with $V\in \sV^{\circ, \epsilon}$ and $\sg$ an isomorphism class of relevant $p$-adic automorphic representations of $\H_{V}(\A)$. 

The relation 
\beq\lb{th s}
\Hom_{\G_{V}(\A^{\infty})\times \H_{V}(\A^{\infty})}(\pi^{\infty,\vee}\ot \omega_{V}^{\infty} \ot \sg^{\infty} , \Qpb^{})\neq 0\eeq
defines a bijection between $\sR_{\Pi , \G}$ and $\sR_{\Pi, \H}$. 
\end{prop}
\begin{proof} 
Take any $\iota\colon \Qpb\into \mathbf{C}$. After base-change to $\mathbf{C}$ via $\iota$, given $\pi$, the existence of $V$ with $\eps(V)=\epsilon(1/2, \Pi)$ and of a representation $\sg^{\infty, \iota}$ of $\H_{V}(\A^{\infty})$ satisfying \eqref{th s} 
  follows from the explicit form of theta dichotomy in terms of the doubling epsilon factors of \cite{harris}, whose  product over all places coincides with the standard central epsilon factor of $\Pi$ by  \cite{LR05}. Again by \cite[Proposition C.3.1]{LTXZZ}, we have that $\sg^{\infty, \iota}$ is the finite component of relevant automorphic representation $\sg^{\iota}$; and as in Remark \ref{BC p} (c), the collection $\sg^{\iota}$ arises from a relevant $p$-adic automorphic representation $\sg$.  
  
  The bijective property of the resulting map $\sR_{\G}\to \sR_{\H}$  follows from \cite[Theorem 4.1 (iv)]{GI16} and the following archimedean fact (see \cite{NY} or \cite[Theorem 4.1 (4)]{PT}):  if  $v\vert \infty $ and   $\pi_{v}$ is the holomorphic discrete series of $U({n\over 2}, {n \over 2})$ with Harish--Chandra parameter $\{\tfrac{n-1}{2},\tfrac{n-3}{2},\dots,\tfrac{3-n}{2},\tfrac{1-n}{2}\}$, then $\pi_{v}$ has a nonzero theta lift to  $H_{V_{v}}$, with $V_{v}\in \sV_{v}$,  exactly for $V_{v}$ positive-definite, in which case  the theta lift $\sg_{v}$ is the trivial representation of $H_{V_{v}}$.
  \end{proof}


\section{Theta cycles}\lb{sec 4}

\subsection{Assumptions on the Galois representation}\lb{sec:ass}
Let again $\rho\colon G_{E}\to \GL_{n}(\Qpb)$ be irreducible, geometric, and of weight $-1$. We  denote by $\rho^{\rc}\colon G_{E}\to \GL_{n}(\Qpb)$ the representation defined by $\rho^{\rc}(g)=\rho(cgc^{-1})$, where $c\in G_{E}$ is any  fixed lift of $\rc$. (A different choice of lift would yield an isomorphic representation.)

We suppose from now on that the following conditions are satisfied:
\begin{enumerate}
\item \lb{csy}
$\rho $ is \emph{conjugate-symplectic} in the sense that there exists a perfect pairing
$$\rho\ot_{\Qpb}\rho^{\rm c}\to \Qpb(1)$$
such that for the induced map $u\colon \rho^{c}\to \rho^{*}(1)$ (where ${}^{*}$ denotes the linear dual) and its conjugate-dual $u^{*}(1)^{\rc}\colon \rho^{\rc}\to \rho^{\rc, *}(1)^{\rc}=\rho^{*}(1)$, we have $u=-u^{*}(1)^{\rc}$;
\item \lb{parity} $n=2r$ is even;
\item\lb{HT} for every place $w\vert p$ of $E$ and every embedding  $\jmath  \colon E_{w}  \into \mathbf{C}_{p}$, the $\jmath$-Hodge--Tate weights\footnote{Our convention is that  the cyclotomic character has weight~$-1$.}
of $\rho$ are the $n$ integers $\{-r, -r+1, \ldots,r-1\}$;
\item \lb{auto}
$\rho$ is \emph{automorphic} in the sense that for each $\iota\colon \Qpb \into \mathbf{C}$, there is a cuspidal automorphic representation $\Pi^{\iota}$ of $\GL_{n}(\A_{E})$ such that $L_{\iota}(\rho, s)=L(\Pi^{\iota}, s+1/2)$.
\end{enumerate}

\subsubsection{Associated automorphic representations}
A collection $(\Pi^{\iota})_{\iota\colon \Qpb\into\mathbf{C}}$ as in Condition \ref{auto} is uniquely determined up to isomorphism if it exists, by the multiplicity-one theorem for automorphic forms on $\GL_{n}$; it is conjectured to always exist. Moreover, every $\Pi^{\iota}$ is relevant in the sense of Definition \ref{defR1}.1, where Condition \ref{csy} implies property (i) in the definition, and Condition \ref{HT} implies property (ii). It is then clear that $(\Pi^{\iota})_{\iota}$ arises from a unique (up to isomorphism) relevant $p$-adic automorphic  representation 
$$\Pi=\Pi_{\rho}$$
of $\GL_{n}(\A_{E})$ (Definition \ref{defR2}.1). We denote by $\pi=\pi_{\rho}$ 
 the relevant $p$-adic representation of $\G(\A)$ associated with $\Pi$ as in Proposition \ref{descent},\footnote{As noted in Remark \ref{can psi}, any other relevant $\pi$ with ${\rm BC}(\pi)=\Pi$ would be equally good.} and by $$(V, \sg)=(V_{\rho}, \sg_{\rho})$$ the pair associated with $\pi$ as in Proposition \ref{th incoh}.  We also put $\H=\H_{V}$.

\subsection{Models of the representations} \lb{models}
We now fix some concrete models of the  representations $\omega$, $\pi$, and $\sg$. 

\subsubsection{Weil representations} 
We fix the well-known model of
 the  representation
 $\omega=\ot'_{v\nmid\infty}\omega_{V,v}$ on  $\sS(V^{r}_{\A^{\infty}}, \Qpb^{})$ associated with $\psi$, on which  $\H(\A^{\infty})$ acts by right translations, whereas the action of $\G(\A^{\infty})$ is recalled in \cite[\S 4.1 (H7)]{DL}.

Denote by   $\dag$ the involution on $\G$ given by conjugation by the element $\smalltwomat {1_{r}}{}{}{-1_{r}}$ inside $\GL_n(E)$; it acts on any $\G(R)$-module for any $E$-algebra $R$. The representation $\omega^{\dag}$ is a model of the Weil representation attached to $\psi^{-1}$.


\subsubsection{Siegel-hermitian modular forms and their $q$-expansion}
The  representation $\pi$ may be realised in spaces of hermitian modular forms, which we briefly review.

In \cite[\S~2.2]{DL}, we have defined the following objects.\footnote{In this discussion, most new notation will be introduced by equalities whose right-hand sides reproduce the corresponding notation in \cite{DL}.} 
\begin{itemize}
\item
A $\mathbf{C}$-vector  space $\sH_{\mathbf{C}}=\sA^{[r]}_{r, {\rm hol}}$ of holomorphic forms for the group $\G$.
\item For any $\Qpb$-algebra $R$, an  $R$-module $\sH_{R}=\sH^{[r]}_{r}\ot_{\Q_{p}}R $ of (classical) $p$-adic automorphic forms for $\G$, such that for each $\iota\colon \Qpb\into \mathbf{C}$, we have an isomorphism 
$$\sH_{\Qpb}\ot_{\iota}\mathbf{C}\to \sH_{\mathbf{C}}, \qquad \Phi\ot 1\mapsto \Phi^{\iota}.$$

In fact, only the case where $E/F$ is totally split above $p$ was considered in \cite{DL}, where $\sH^{[r]}_{r}$ is the direct limit, over open compact subgroups $U\subset \G(\A^{\infty})$, of subspaces of sections of a certain line bundle on a Siegel hermitian variety $\Sigma(U)_{/\Q_{p}}$; let us explain why the splitting condition is not necessary for our purposes. Define a \emph{$p$-adic CM type} of $E$ to be a set $\Phi$ of $[F:\Q]$ embeddings $i\colon E\into \Qpb$ such that $i\in \Phi$ if and only if $i\circ \rc\notin \Phi$; in the totally split case, the choice of a $p$-adic CM type is equivalent to the choice of a set $\mathtt{P}_{\rm CM}$  as in \cite[\S 2.1 (F2)]{DL}, which intervenes in the construction of $\Sigma(U)$ as a moduli scheme by fixing a \emph{signature type} for test objects in the sense of \cite[Definition 3.4.3]{LTXZZ}. However, this construction, and the comparison with complex Siegel hermitian varieties of \cite[Lemma 2.1]{DL}, go through with any $p$-adic CM type $\Phi$ (with the innocuous difference that, in general, $\Sigma(U)$ and $\sH_{r}^{[r]}$ will only be defined over a finite extension of $\Q_{p}$ in $\Qpb$).

\item  A space ${\rm SF}_{R}={\rm SF}_{r}(R)$ of formal $q$-expansions with coefficients in the (arbitrary) ring $R$, and a Siegel--Fourier expansion map ${\bf q}_{\infty}={\bf q}_{r}^{\rm an}\colon \sH_{\mathbf{C}}\to {\rm SF}_{\mathbf{C}}$.  By the argument at the end of the proof of \cite[Proposition 4.18]{DL} (based on Lemma 2.11 \emph{ibid.}), we deduce a $\Qpb^{}$-linear $q$-expansion map 
$$\qqq_{p}\colon \sH_{\Qpb}\to {\rm SF}_{\Qpb}$$ satisfying 
$\iota \qqq_{p} (\Phi) = \qqq_{\mathbf{C}}(\Phi^{\iota})$ for every $\Phi\in \sH_{\Qpb^{\cyc}}$ and  every  embedding $\iota\in \Sg$. 
\end{itemize}

By \cite[Lemma 3.14]{DL} (based on \cite{Mok}), for a relevant $p$-adic automorphic representation $\pi$, the space $\Hom_{\G(\A^{\infty})}(\pi, \sH_{\Qpb})$ is $1$-dimensional, and $\pi^{\vee, \dag}$ is also relevant.  We identify $\pi=\pi_{\rho}$ with the corresponding subspace of $\sH_{\Qpb}$. Then $\pi_{\rho^{*}(1)}$ is isomorphic to $\pi^{\vee, \dag}$. 

Moreover, for any ring $R$, let $\ul{\r{SF}}_{\, R}$ be the space of those formal expansions
$$\sum_{T\in \Herm_{r}(F)^{+}} c_{T}(a) \, q^{T}, \qquad c_{T}\in C^{\infty}( \GL_{r}(\A^{\infty}_{E}) , R)$$
satisfying $c_{{}^{\rt}a^{\rc}Ta}(y)=
c_{T}(ay)$ for all $a\in \GL_{r}(E)$; then we have a $q$-expansion map 
$$\ul{\qqq}_{p}\colon \sH_{\Qpb}\to \ul{\r{SF}}_{\, \Qpb}$$
characterised by
 $\ul{\qqq}_{p}\Phi(y)  = |\det y|_{E}^{r}\qqq(m(y)\Phi)$. Since $\rM(\A^{\infty})$ acts transitively on the set of connected components of $\Sg(U)_{\Qpb}$ for every open compact subgroup $U\subset \G(\A^{\infty})$, the map $\ul{\qqq}_{p}$ is injective.


\subsubsection{Shimura varieties and their cohomology}  We assume from now on that $\eps(\rho)=-1$. (The opposite case will be  trivial for our purposes in Definition \ref{def} below.)  Then $V\in \sV^{\circ, -}$, and we have an inverse system   $$(X_{K})_{K\subset \rH(\A^{\infty})}$$ of $(n-1)$-dimensional smooth varieties over $E$,  with the property that for every archimedean place $w$ of $F$, with underlying place $v$ of $F$, the variety $X_{V,K}\times_{E, w}\mathbf{C}$ is isomorphic to the complex Shimura variety $X_{V(v),w K}$  associated with the unitary group $\H_{V(v)}$ and the Shimura datum attached to $w$ that is the complex conjugate to the one defined in \cite[\S~C.1]{liu-fj} (and thus coincides with the one specified in \cite[\S~3.2]{LTXZZ} and used in \cite{LL, DL});  see also \cite{gross-incoh, STay}.   

From now on we assume  that each $X_{K}$ is projective, which is the case if  and only if either $F\neq \Q$, or $n=2$,  $F=\Q$ and $\eps(V_{v})=-1$ for some finite place~$v$. In fact, in the remaining non-compact case for $n=2$, the curve $X_{K}$ (closely related to a classical  modular curve) can be  canonically compactified by adding finitely many cusps; in this case the constructions make sense, and the theorems hold true, after replacing $X_{K}$ by its compactification.

Let
\beqq
H_{\textup{\'et}}^{2r-1}(X_{\ol{E}}, \Qpb(r))&\coloneqq  \varprojlim_{K\subset \H(\A^{\infty})} H_{\textup{\'et}}^{2r-1}(X_{K, \ol{E}}, \Qpb(r)),
\eeqq
where the transition maps are  pushforwards.
For each $K$, we have a spherical Hecke algebra for $\H$ acting on $X_{K}$; let $\frakm_{\rho, K}$ be the Hecke ideal denoted by $\frakm^{\mathtt{R}}_{\pi}$ in \cite[Definition 6.8]{LL}. We denote by
$$M_{\rho, K}\coloneqq   H_{\textup{\'et}}^{2r-1}(X_{K, \ol{E}}, \Qpb(r))_{\frakm_{\rho,K}}$$
the localisation, and we set 
\beqq
M_{\rho}\coloneqq \varprojlim_{K} M_{\rho, K}\subset H_{\textup{\'et}}^{2r-1}(X_{\ol{E}}, \Qpb(r)).\eeqq

 We will assume the following hypothesis, which is a special case of \cite[Hypothesis 6.6]{LL}  (it is known for $n=2$, and  it is expected to be confirmed in general in a sequel to \cite{KSZ}).
\begin{enonce}{Hypothesis}   \lb{hyp coh}
For each open compact $K\subset \H(\A^{\infty})$,  we have a Hecke- and Galois-equivariant decomposition
\beq \lb{eq coh}
M_{\rho, K}\cong \bigoplus_{\sg'} \rho \ot \sg'^{\vee, K},\eeq
where the direct sum runs over the isomorphism classes of  relevant $p$-adic automorphic representation $\sg'$ of $\H_{V}(\A)$ with ${\rm BC}(\sg')=\Pi$.
\end{enonce}
 We thus have an $\H(\A^{\infty})$-equivariant map
\beq\lb{why push}
\sg \longrightarrow \Hom_{\Qpb[G_{E}]} ( H_{\textup{\'et}}^{2r-1}(X_{\ol{E}}, \Qpb(r)), \rho),\eeq
and we identify $\sg$ with the image of this map.  We also put $M_{\sg, K}\coloneqq  \rho\ot\sg^{\vee, K}\subset H_{\textup{\'et}}^{2r-1}(X_{K, \ol{E}}, \Qpb(r))$, and 
\beq
\lb{Msg}
M_{\sg}\coloneqq \varprojlim_{K} M_{\sg, K}\subset M_{\rho}\subset H_{\textup{\'et}}^{2r-1}(X_{\ol{E}}, \Qpb(r)).\eeq
Then $\sg=\Hom_{\Qpb[G_{E}]}(M_{\sg}, \rho)\coloneqq  \varinjlim_{K} \Hom_{\Qpb[G_{E}]}(M_{\sg, K}, \rho)$.

Denote by ${\rm Fil}^{\bullet}\subset H^{2r}_{\text{\'et}}(X_{K}, \Q_{p}(r)) $ the filtration induced by the Hochschild--Serre spectral sequence $H^{i}(E, H^{2r-i}_{\text{\'et}}(X_{K}, \Q_{p}(r))) \Rightarrow  H^{2r}_{\text{\'et}}(X_{K}, \Q_{p}(r)) $.  By the argument for  \cite[Lemma 4.7]{DL}, we have a  canonical  Hecke-equivariant  projection
$$
H^{2r}_{\text{\'et}}(X^{}_{K}, \Qpb(r)) /\r{Fil}^{2}) 
\to H^{1}(E_{},  M_{\rho, K}).$$
\begin{lemm} The image of the composition 
$$[-]_{\rho}\colon \Ch^{r}(X_{K})^{0}_{\Qpb}\stackrel{\rm AJ}{\longrightarrow} H^{2r}_{\textup{\'et}}(X_{K}, \Qpb(r)) /\r{Fil}^{2}
\to H^{1}(E_{},  M_{\rho, K})$$
is contained in $ H^{1}_{f}(E_{},  M_{\rho, K})$
\end{lemm}
\begin{proof}
As in \cite[Lemma 4.24]{DL}, using \cite[Theorem B]{nek-niz} in place of \cite{nek-AJ} for $p$-adic places.
\end{proof}
  
 \subsection{Construction} \lb{const}
 We proceed in four steps. The first three  steps follow  works of Kudla and collaborators \cite{Kud-Duke, Kud03, KRY06}, and of Liu and collaborators \cite{Liu11, DL}.
 
\subsubsection{0. Special cycles in $X$} For each $x\in V^{r}_{\A^{\infty}}$ and each open compact $K\subset \H(\A^{\infty})$, we have a codimension-$r$ special cycle 
$$Z(x)_{K}\in \Ch^{r}(X_{K})$$
defined in \cite[\S~3A]{Liu11}. Putting $$T(x)\coloneqq (( x_{i}, x_{j})_{V})_{ij},$$ where $(\, ,\, )_{v}$ is the hermitian form on $V$, we recall the definition in two  basic cases.  Denote by $\Herm_{r}(F)^{+}$ the set of $r\times r$ matrices over $E$ that satisfy $T^{\rc}=T^{\rt}$ and that are totally positive semidefinite.   First,  $Z(x)_{K}=0$ if  $T(x)\notin \Herm_{r}(F)^{+}$. Second, assume that $T(x) \in  \Herm_r(F)^{+}$ is positive {definite}. Let $V_{x}\subset V$ be the incoherent hermitian space that is (place by place) the orthogonal complement of the span of $(x_{1},\ldots, x_{r})$. The corresponding embedding   $\rU(V_{x})\into \rU(V)$  of incoherent unitary groups
 induces a map of towers of Shimura varieties $\alpha_{x} \colon X_{V_{x}}\to X_{V}$; then we define $Z(x)_{K}\in \Ch^{r}(X_{V, K})$ to be the class of the image cycle.
%

\subsubsection{1. Theta kernel}  The special cycles just defined may be assembled into a generating series.
Let $\phi\in \omega$. For every $K\subset \H_{V}(\A^{\infty})$ fixing $\phi$, we define
$${}^{\qqq}\Theta(\phi)_{\rho, K}(a) \coloneqq \vol(K) \sum_{{x\in K\bks  V_{\A^{\infty}}^{r}}} \phi(xa) [Z(x)_{K}]_{\rho} \, q^{T(x)},$$ 
where $\vol(K)$  is as in \cite[Definition 3.8]{LL}. 
Then ${}^{\qqq}\Theta(\phi)_{\rho, K}$ is  an element of $  H^{1}_{f}(E, M_{\rho, K})\ot_{\Qpb}\ul{\rm SF}_{\Qpb}$, and the construction is compatible under pushforward in the tower $X_{ K}$.
(The reason why we prefer our $\Theta(\phi)_{-, \rho} $ to be compatible with pushforwards rather than pullbacks is that this allows to pair them, in Step 3, with elements of the automorphic representation $\sg$ under the identification \eqref{why push}.)

The following conjecture, which is a variant of \cite[Hypothesis 4.16]{DL}, asserts the modularity of the generating series, and from now on we will assume it holds. 
\begin{enonce}{Hypothesis}  \lb{hyp mod}
For every $\phi\in \omega$ and any  $K\subset \H_{V}(\A^{\infty})$ fixing $\phi$,   there exists a unique 
$$\Theta(\phi)_{\rho, K} \in 
  H^{1}_{f}(E, M_{\rho, K})\ot_{\Qpb} 
 \sH_{\Qpb^{}}$$
such that 
$$\ul{\qqq}_{p}(\Theta(\phi)_{K,\rho}) =  {}^{\qqq}\Theta(\phi)_{\rho, K}.$$
\end{enonce}

\begin{rema}
A recent piece of evidence for this modularity conjecture is  provided in \cite[Theorem 4.20]{DL}, which is recalled as part of Theorem \ref{ht f}.2; moreover,\footnote{I am grateful to Yifeng Liu for bringing this to my attention.}  an analogous conjecture for orthogonal Shimura varieties can be deduced from \cite{kudla-mod}.
Hypothesis \ref{hyp mod} is implied by the variant for Chow groups of \cite[Hypothesis 4.5]{LL}. See Remark 4.6  \emph{ibid.} for comments on the supporting evidence for that conjecture until then, to which we should add the recent \cite{Xia}. For the history, which traces back to the work of Gross--Kohnen--Zagier on  generating series of Heegner points \cite{GKZ}, see  \cite[Remark 3.5.5]{Chao}, cf. also \emph{ibid.} \S~6.4. 
\end{rema}

\subsubsection{2. Arithmetic theta lifts}
Denote by $\Phi\mapsto \Phi_{\pi}$ the Hecke-eigenprojection $\sH_{\Qpb}\to \pi$, and by  $\lan\, , \, \ran_{\pi^{\vee}}\colon \pi^{\vee}\ot\pi\to \Qpb$ the canonical duality. (We also use the same names for any base-change.)

Then for every $\vphi\in \pi^{\vee}$, we may define  
\beq \lb{ATL}
\Theta(\vphi, \phi)_{K}\coloneqq  \lan \vphi, \Theta(\phi)_{K, \rho, \pi}\ran_{\pi^{\vee}}\quad \in   H^{1}_{f}(E, M_{\rho, K}).\eeq
Since the map $(\vphi, \phi)\mapsto \Theta(\vphi, \phi)_{K}$ is equivariant under the action of $\Qpb[K\bks \H_{V}(\A^{\infty}) /K]$, Proposition \ref{th incoh} implies that $\Theta(\vphi, \phi)_{K}$ belongs to the subspace $H^{1}_{f}(E, M_{\sg, K})\subset H^{1}_{f}(E, M_{\rho, K})$.

\subsubsection{3. Theta cycles}
For every $f\in \sg$,$\vphi\in \pi^{\vee}$, and any 
 $K\subset \H_{V}(\A^{\infty})$ fixing $f $ and $\phi$, 
we define 
$$\Theta_{\rho} (\vphi, \phi, f)\coloneqq  f_{*}\Theta(\vphi, \phi)_{K}
 \in H^{1}_{f}(E, {\rho})^{}. $$

The following definition then satisfies the first property asserted in Theorem \ref{main}. 

\begin{defi}\lb{def} Let $\rho$ be a  Galois representation satisfying the assumptions of \S~\ref{sec:ass}. 

If $\eps(\rho)=+1$, we may put $\Lm_{\rho}=\Qpb$ and $\Theta_{\rho}\coloneqq 0$. 

If $\eps(\rho)=-1$, assume that $F\neq \Q$ and that Hypotheses \ref{hyp coh} and \ref{hyp mod} hold, and let $\pi$, $V$, $\sg$ be as above. Then we define 
\beqq
\Lm_{\rho}\coloneqq  (\pi^{\vee}\ot\omega \ot \sg)_{\G(\A^{\infty})\ts \H(\A^{\infty})},
\eeqq
and \beqq  \Theta_{\rho}\colon \Lm_{\rho}&\to  H^{1}_{f}(E, {\rho}),\\
[(\vphi, \phi, f)] &\mapsto \Theta_{\rho}(\vphi, \phi, f). \eeqq
\end{defi}


\begin{rema} Suppose that $n=2$ and that $\rho=V_{p}A_{E}$ for a modular abelian variety $A$ of $\GL_{2}$-type over $F$. Then the image of $\Theta_{\rho}$ consists of classes of Heegner points. This follows by comparing the height formulas for the two objects in  \cite{YZZ} and \cite{LiuII}, against the backdrop of \cite{nek-CM}. A direct comparison is also possible:  for $n=2$,  all the $Z(x)$ are CM points on unitary Shimura curves, which can be related along the lines of \cite[\S~4]{carayol} 
to the modular curves and the quaternionic Shimura curves used to construct Heegner points in  \cite{GZ, YZZ}. 
\end{rema}

 \section{Relation to $L$-functions and Selmer groups}\lb{sec 5}
We continue to denote by $\rho $ a  Galois representation satisfying the assumptions of \S~\ref{sec:ass}. 

\subsection{Complex and $p$-adic $L$-functions}
For every $\iota\colon \Qpb\into \mathbf{C}$, and every finite-order character $\chi'\colon G_{E}\to \Qpb^{\ts}$, we have the $L$-function 
$$L_{\iota}(\rho\ot\chi', s)= L(s+1/2, \Pi^{\iota}\ot\iota\chi'), $$
which is holomorphic and has a functional equation with center at $s=0$ and sign $\eps(\rho)$.

At least under the following assumption, we also have a $p$-adic $L$-function.  
\begin{enonce}{Assumption} \lb{ass p}
The extension $E/F$ is totally split above $p$, and  for every place $w\vert p$ of $E$, the representation $\rho_{w}$ is crystalline and Panchishkin-ordinary.
\end{enonce}
We  need to make the auxiliary choice of an isomorphism $\al \colon \pi^{\vee, \dag}\to \pi_{\rho^{*}(1)}$ (where $\pi=\pi_{\rho}, \pi_{\rho^{*}(1)}\subset\sH_{\Qpb}$), which yields for each $\iota\colon \Qpb\into \mathbf{C}$, an element
 $\rP_{\rho, \iota}=\rP_{\rho, \alpha, \iota}(\rho)\in \mathbf{C}^{\ts}$ such that
\beqq
\iota ( \vphi_{1}^{\dag} , \vphi_{2})_{\pi^{\vee}} ={ (  (\al \vphi_{1})^{\iota, \dag},  \vphi^{\iota}_{2} )_{\rm Pet}
 \over \rP_{\rho, \iota}}\eeqq
for every $\vphi_{1}\in \pi^{\vee, \dag}$, $\vphi_{2}\in \pi$; here
$$(\vphi, \vphi')_{\rm Pet}\coloneqq  {\int_{\G(F)\bks \G(\A)}   \vphi^{} (g) \vphi'(g)\, dg}$$
where
 $dg$ is  the measure of \cite[\S~2.1 (G7)]{DL}. 

For a character $\chi$ of $G_{F}$, we put  $\chi_{E}\coloneqq \chi_{|G_{E}}$, and  $b_{n}(\chi)\coloneqq \prod_{v\nmid \infty}b_{n}(\chi_{v})$, where the factors are as in \eqref{bnchi};  we also  define a constant
$$
c_{\infty}=\left(
(-1)^{r}2^{-r^{2}-r} \pi^{r^{2}}{\Gamma(1)\cdots\Gamma(r)\over \Gamma(r+1)\cdots \Gamma(2r)}
\right)^{[F:\Q]}.$$
Finally, we denote by $\sK(\sX_{F})$ the fraction field of $\sO(\sX_{F})$.

\begin{prop} Suppose that $\rho $ satisfies Assumption \ref{ass p}. There is a meromorphic function $$L_{p}(\rho)=L_{p, \al}(\rho)\quad \in  \sK(\sX_{F})$$
 characterised by the following property: for every finite-order character $\chi\in \sX_{F}(\Qpb)$ and every embedding $\iota\colon \Qpb\into \mathbf{C}$, we have
$$\iota L_{p}(\rho)(\chi) = \iota e_{p}(\rho, \chi)\cdot   { c_{\infty} L_{\iota}( \rho\ot\chi_{E}, 0) \over  b_{n}(\chi) \rP_{\rho, \iota}}.$$  
Here, $\iota e_{p}(\rho, \chi)= 
\prod_{w|v\vert p} \iota e_{w, \iota}(\rho, \chi)\in \iota\Qpb$, in which  the product ranges over the $p$-adic  places of $E$ and of $F$, and
 $$\iota e_{w}(\rho, \chi) \coloneqq 
 \gamma(\iota {\rm WD}(\rho_{w}^{+}\ot\chi_{E,w}),\psi_{E,w})^{-1} {b_{n, v}(\chi)\over L_{\iota}(\rho_{w}\ot\chi_{E, w})}.$$
 where  the Deligne--Langlands $\gamma$-factor and  Fontaine's functor  $\iota{\rm WD}$ are as recalled in \cite[(1.1.4)]{Dplf}. 
\end{prop}
\begin{proof} This follows by multiplying the incomplete $p$-adic $L$-function of \cite[Theorem 1.4]{DL}  by local $L$-factors at ramified and $p$-adic places, as in the proof of Proposition 3.39 \emph{ibid.}\footnote{Before  \cite{DL}, a $p$-adic $L$-function that extends $L_{p}(\rho)$ to a larger space was constructed in \cite{EHLS};  the  rationality property proved there is weaker than stated here.}  
\end{proof}

\subsection{Pairings}
Let $\rho$ be a representation satisfying the assumptions of Definition \ref{def}, and let  $\pi_{\rho}$, $V$, $\sg_{\rho}$, $\Lambda_{\rho}$, and $\Theta_{\rho}$ be the associated objects. We denote by $\pi_{v}$ and $\sg_{v}$ the local components of $\pi_{\rho}$ and $\sg_{\rho}$ at the place $v$ (which are well-defined up to isomorphism).

\subsubsection{Dual Theta cycles} 
The representation $\rho^{*}(1)$ also satisfies those assumptions, and we have
%
the corresponding  map
 $$\Theta_{\rho^{*}(1)}^{}\colon \Lm_{\rho^{*}(1)}\to H^{1}_{f}(E, \rho^{*}(1)).$$

\subsubsection{Pairings} 

Let  $\lan \, , \, \ran \colon M_{\rho}\ot M_{\rho^{*}(1)}\to \Qpb(1)$ be the pairing induced by Poincar\'e duality. Then we define a pairing 
\beq \lb{pair sg}
(\, , \, )_{\sg}\colon \sg_{\rho}\ot \sg_{\rho^{*}(1)} \to \Qpb\eeq
 by $(f, f')_{\sg}\coloneqq  f\circ u( f'^{*}(1))$, where $f'^{*}(1)\colon \rho_{\sg^{*}(1)}^{*}(1)\to M_{\rho^{*}(1)}^{*}(1)$ is the transpose, and $u\colon M_{\rho^{*}(1)}^{*}(1)\to M_{\rho}$ is the isomorphism induced by $ \lan\, , \, \ran$. 
 Thus $\sg_{\rho^{*}(1)}$ is identified with $\sg_{\rho}^{\vee}=\sg^{\vee}$.
 
 We also have a canonical pairing on $\omega\ot \omega^{\dag}$ defined by
\beq\lb{pair omega}
(\phi, \phi')_{\sg}=\int_{V^{r}_{\A^{\infty}}}\phi(x)\phi'(x)\, dx\eeq
for the product of $\psi$-selfdual measures. Thus $\omega^{\dag}$ is identified with $\omega^{\vee}$. Similarly, if we denote  $\iota\Box\coloneqq \Box\ot_{\Qpb, \iota}\mathbf{C}$, and complex conjugation in $\mathbf{C}$ by a bar, we have $\ol{\iota\omega}=\omega^{\vee}$. Let $\vol(H_{\infty})$ be the volume of $\H(F_{\infty})$ for the measure denoted ${1\over b_{2r}(0)}\, dh_{v}^{\natural}$ in \cite[Definition 3.8]{LL}, which is a rational number by \cite[Lemma 2.2.1]{DZ}.

Then:
 \begin{itemize}
 \item
 for every isomorphism $\al\colon \pi_{\rho}^{\vee, \dag}\to \pi_{\rho^{*}(1)}$, we
  have a pairing
\beq\lb{z al}
\zeta_{\al}\coloneqq  \vol(H_{\infty})\cdot \ot_{v\nmid \infty}\zeta_{v}\circ (\, )^{\dag}\circ j_{\al} \colon \Lm_{\rho}\ot\Lm_{\rho^{*}(1)}\to \Qpb,\eeq
where $j_{\al}$ identifies the factor $\pi_{\rho^{*}(1)}^{\vee}$ of $ \Lm_{\rho}\ot\Lm_{\rho^{*}(1)}$ with $\pi_{\rho}^{\dag}$ via the dual of $\al$, and $(\, )^{\dag}$ maps $\pi_{\rho}^{\dag}\ot \omega$ to $\pi_{\rho}\ot \omega^{\dag}=\pi_{\rho}\ot \omega^{\vee}$;

\item for every $\iota\in \Sg$ we have an identification $j_{\iota}\colon \iota{\pi_{\rho^{*}(1)}^{\vee} \stackrel{\cong}{\longrightarrow} \ol{\iota\pi_{\rho}}}$
via the restriction of $(\, , \, )_{\rm Pet}$ to $\ol{\pi_{\rho}^{\iota}}\ot\pi_{\rho^{*}(1)}$. Then 
 we obtain a pairing
$$\zeta_{\iota} \coloneqq  \vol(H_{\infty})\cdot   \ot_{v\nmid \infty}\zeta_{v}\circ \ol{(\, ) }\circ j_{\iota}\colon \iota\Lm_{\rho}^{}\ot\iota\Lm_{\rho^{*}(1)}\to \mathbf{C}$$
where $\ol{(\, )}$ maps $\ol{\iota\pi_{\rho}}\ot\iota\omega $ to $ \iota\pi_{\rho}\ot \ol{\iota\omega}= \iota\pi_{\rho} \ot \iota\omega^{\vee}$. 
  \end{itemize}

\subsubsection{$p$-adic height pairing} 
Assume that $\rho $ is Panchishkin-ordinary. 
Then the construction of \nek\ \cite{nek-height} (see \cite[\S~4.2]{DL} for a verification of the assumptions) yields  a $p$-adic height pairing 
$$\lan \  ,\  \ran\colon H^{1}_{f}(E, \rho) \ot H^{1}_{f}(E, \rho^{*}(1))\to \Gamma_{F}\hat{\ot}\Qpb.$$ 

\subsubsection{Complex height pairings}
On the other hand, assume that $p$ is unramified in $E$,  and let  $K_{p}^{\circ}=\prod_{v\vert p}H_{v}\subset H_{p}$ be a product ot maximal hyperspecial subgroups. Then for open compact $K^{p}\subset \H(\A^{p\infty})$, setting $K\coloneqq K^{p}K_{p}^{\circ}\subset \H(\A^{\infty})$, the variety $X_{K}$ has good reduction at all $p$-adic places.  Define
$$ \Ch^{r}(X_{K})^{\lan p\ran}\subset \Ch^{r}(X_{K})^{0} $$
to be the $\Q$-subspace of algebraic cycles whose class in $H^{2r}(X_{K, E_{w}}, \Q_{p}(r))$ is trivial for every finite place $w\nmid p $ of $E$.
Li and Liu \cite{LL} observed that the construction of  Be\u\i linson \cite{Bei87} unconditionally defines a height pairing 
\beq \lb{eq BB}
\lan\, ,\,  \ran^{\rm BB}\colon \Ch^{r}(X_{K})_{\mathbf{C}}^{\lan p\ran} \ot_{\mathbf{C}} \Ch^{r}(X_{K})_{\mathbf{C}}^{\lan p \ran}\to \mathbf{C}\ot_{\Q}\Q_{p}\eeq
that is $\mathbf{C}$-linear in the first factor and $\mathbf{C}$-antilinear in the second factor.  (It is conjectured that the pairing takes values in $\mathbf{C}\subset \mathbf{C}\ot_{\Q}\Q_{p}$; this turns out to be the case in the application to Theta cycles.)

In order to descend this pairing to Selmer groups, we need to assume a case of a standard conjecture on the injectivity of Abel--Jacobi maps. Whenever $K\subset \H(\A^{\infty})$ is an open compact subgroup that is understood from the context, denote $\frakm_{\rho}=\frakm_{\rho, K}$, $\frakm_{\rho^{*}(1)}=\frakm_{\rho^{*}(1), K}$
 
\begin{conj}
\lb{ass infty}
For $\rho^{?}\in \{\rho, \rho^{*}(1)\}$ and for each open compact  subgroup $K^{p}\subset \H(\A^{p\infty})$, 
the Abel--Jacobi map 
\beq \lb{AJc}
{\rm AJ}_{p,  K^{p}K_{p}^{\circ}}\colon
 \left(\Ch^{r}(X_{K^{p}K_{p}^{\circ}})_{\Qpb}^{\lan p \ran}\right)_{\frakm_{\rho^{?}}}\to H^{1}_{f}(E, M_{\rho^{?}, K^{p}K_{p}^{\circ}  })
 \eeq
 is injective. 
\end{conj}

Assume that $\rho$ is crystalline at all $p$-adic places. Fix a maximal hyperspecial subgroup $K_{p}^{\circ}\subset \H(\A^{p\infty})$, and assume that Conjecture \ref{ass infty} holds. Denote by   $H^{1}_{f}(E, M_{\rho^{?}, K^{p}K_{p}^{\circ}})^{X}$ the image of \eqref{AJc}, and let 
$$H^{1}_{f}(E, \rho^{?})^{X_{K_{p}^{\circ}}}\coloneqq \sum_{\sg', K^{p}}\sum_{f'\in (\sg')^{K^{p}K_{p}}} f'_{*} H^{1}_{f}(E, M_{\rho^{?}, K^{p}K_{p}^{\circ}})^{X}$$
where the first sum  is as in  \eqref{eq coh} for $\rho^{?}$. 
Then for every $\iota\colon \Qpb\into \mathbf{C}$ and every $K=K^{p}K_{p}^{\circ}$, we have a pairing 
\beq \lb{ht on X}
\lan\ , \ \ran^{\iota}_{K}\colon H^{1}_{f}(E, M_{\rho, K})^{X} \ot_{\Qpb}
 H^{1}_{f}(E, M_{\rho^{*}(1), K})^{X} \ot_{\Qpb, \iota}\mathbf{C}\to \mathbf{C}\ot_{\Q}\Q_{p}\eeq
 transported from \eqref{eq BB} via the maps ${\rm AJ}_{p,  K^{p}K_{p}^{\circ}}\ot_{\iota}1$. We may deduce from  it a pairing 
\beq \lb{eq BB 2}
\lan\ , \ \ran^{\iota}\colon H^{1}_{f}(E, \rho)^{X_{K_{p}^{\circ}}} \ot_{\Qpb} H^{1}_{f}(E, \rho^{*}(1))^{X_{K_{p}^{\circ}}} \ot_{\Qpb, \iota}\mathbf{C}\to \mathbf{C}\ot_{\Q}\Q_{p}\eeq
defined as follows. 

For $i=1, 2$ let 
$$c_{i}=f_{i, *}{\rm AJ}_{p, K} c'_{i}$$ for some $K=K^{p}K_{p}^{\circ}$,  some $f_{i}\in \sg_{i}^{K}$, and some 
$$c_{1}'\in \left(\Ch^{r}(X_{K^{p}K_{p}^{\circ}})_{\Qpb}^{\lan p\ran}\right)_{\frakm_{\rho}}, 
\qquad c_{2}'\in \left(\Ch^{r}(X_{K^{p}K_{p}^{\circ}})_{\Qpb}^{\lan p \ran}\right)_{\frakm_{\rho^{*}(1)}}.$$ If $\sg_{1}\not\cong \sg_{2}^{\vee}$, we put 
\beq\lb{eq BB 3}
\lan c_{1}, c_{2}\ran^{\iota}\coloneqq 0.\eeq
 If $\sg_{1}\cong\sg_{2}^{\vee}$, we have the pairing $(\, , \ )_{\sg_{1}}$ of \eqref{pair sg} on $\sg_{1}\ot \sg_{2}$, through which we identify $\sg_{2}=\sg_{1}^{\vee}$. 
Let  $$\rt_{K}(f_{1}\ot f_{2})\in \Hom(\sg_{1}^{\vee, K}, \sg_{2}^{K})=\End(\sg_{1}^{\vee, K})
=\End_{\Qpb[G_{E}]}(M_{\sg_{1}, K})$$ be given by 
$$\rt_{K}(f_{1}\ot f_{2})(v_{1})  = \vol(K)\cd (v_{1}, f_{1})_{\sg_{1}} \cd f_{2},$$
and let 
\beq\lb{vol here}
\rt(f_{1}\ot f_{2})(v_{1})  = \vol(K)\cd \rt_{K}(f_{1}\ot f_{2});\eeq
 the normalising volume factor makes ${\rt} $ into a well-defined map $\sg_{1}\ot \sg_{1}^{\vee}\to \End_{\Qpb[G_{E}]}(M_{\sg_{1}})$.  The existence of a Hecke correspondence acting as $\rt(f_{1}\ot f_{2})$ implies that the action of  $\rt(f_{1}\ot f_{2})$ on Selmer groups preserves the subspace $H^{1}_{f}(E, M_{\sg_{1}, K})^{X_{K_{p}^{\circ}}}$. 
Then we define
\beq\lb{eq BB 4}
\lan c_{1}, c_{2}\ran^{\iota}\coloneqq  \lan \rt(f_{1}\ot f_{2}) c_{1}', c_{2}'\ran^{\iota}_{K}.\eeq
The definition of \eqref{eq BB 2} in the general case  follows from \eqref{eq BB 3}, \eqref{eq BB 4} by bilinearity.

\begin{rema} \lb{rm ht}
In the $p$-adic case, we also have $\Gamma_{F}\hat{\ot}\Qpb$-valued \nek\ pairings $ \lan \ ,  \ \ran_{K}$ analogous to \eqref{ht on X} (whose construction takes  as input the pairing on $M_{\rho, K}\ot M_{\rho^{*}(1), K}$ deduced from Poincar\'e duality). The analogous formula to \eqref{eq BB 4} holds true as a consequence of  the definitions and the projection formula \cite[Lemma A.2.5]{DZ}.
\end{rema}


\subsection{The height formulas} 
We may now state the main known results on Theta cycles. They parallel those of \cite{GZ, PR, Kol} on Heegner points.

We will say that $E$ and $\rho$ are \emph{mildly ramified} if  $E$ and $\pi_{\rho}$  satisfy the hypotheses of \cite[Assumption 1.6]{DL}, except possibly for the ones about $p$-adic places.

\begin{theo}\lb{ht f} Suppose that $F\neq \Q$ or $n=2$, that $E$ and  $\rho$ are mildly ramified, and that $\rho$ is crystalline at all $p$-adic places.  Assume  Hypothesis \ref{hyp coh}.
\begin{enumerate}
\item
Assume the Modularity Hypothesis \ref{hyp mod}, Conjecture \ref{ass infty}, and that $p$ is unramified in $E$.   Then for every $\lm \in \Lm_{\rho}$, $\lm' \in \Lm_{\rho^{*}(1)}$ and for every $\iota\in \Sg$, we have
$$\lan \Theta_{\rho} (\lm),
\Theta_{\rho^{*}(1)}(\lm')\ran^{\iota} 
= {c_{\infty}L_{\iota}'(\rho, 0)\over b_{n}(\one)} \cd \zeta_{\iota}(\lm, \lm')
$$
in $\mathbf{C}$.

\item 
Suppose that Assumption \ref{ass p} holds and that $p>n$. Let  $\al\colon \pi_{\rho}^{\vee, \dag}\cong\pi_{\rho^{*}(1)}$.  Then:
\begin{itemize}
\item if  the order of vanishing of $L_{p, \al}(\rho)$ at $\one $ is  one, then the Modularity Hypothesis \ref{hyp mod} holds, and 
 for every $\lm \in \Lm_{\rho}$, $\lm'\in \Lm_{\rho^{*}(1)}$,
  we have
$$\lan \Theta_{\rho} (\lm), 
\Theta_{\rho^{*}(1)}(\lm')\ran 
= e_{p}(\rho, \one)^{-1} \cdot \rd L_{p, \al}(\rho)(\one) \cd \zeta_{\al}(\lm, \lm').
$$
in $\Gamma_{F} \hat{\ot}\Qpb=T_{\one}^{*}\sX_{F}$.
\item   if  the order of vanishing of $L_{p, \al}(\rho)$ at $\one$ is not one and  the Modularity Hypothesis \ref{hyp mod} holds, then
 for every $\lm \in \Lm_{\rho}$, $\lm'\in \Lm_{\rho^{*}(1)}$,
we have
$$\lan \Theta_{\rho} (\lm), 
\Theta_{\rho^{*}(1)}(\lm')\ran =0.$$
\end{itemize}
\end{enumerate}
\end{theo}
\begin{proof}
 Write $\lm =[(\vphi, \phi, f)]$, $\lm'=[( \vphi', \phi', f')]$. 
 Consider the $p$-adic case. The  modularity result  is \cite[Theorem 4.20]{DL}, after projection $H^{1}_{f}(E,M_{\rho}) \to H^{1}_{f}(E, M_{\sg'})$ for any relevant $\sg'$ with ${\rm BC}(\sg')=\Pi$; but this is equivalent to the modularity in $H^{1}_{f}(E, \rho) $ by Hypothesis \ref{hyp mod}.

For the first height formula, by the definitions and Remark \ref{rm ht}, it is equivalent to prove
\beq
\lb{my AIP}
\lan \rt(f\ot f'^{\vee})\, \Theta(\vphi, \phi), 
\Theta(\vphi', \phi') \ran 
= e_{p}(\rho, \one)^{-1} \cdot \rd L_{p, \al}(\rho)(\one) \cd \breve\zeta_{\al}(  \rt(f\ot f'^{\vee})\vth(\vphi, \phi'); \vth(\vphi', \phi'))
\eeq
where the $\Theta$'s are the arithmetic theta liftings for $\rho$ and $\rho^{*}(1)$ as in \eqref{ATL}, and  
$$\breve\zeta_{\al} =\vol(H_{\infty}) \cdot \ot_{v\nmid \infty}\breve\zeta_{v} \circ (\ )^{\dag}\circ j_{\al}$$ is defined analogously to \eqref{z al} based on  the pairings \eqref{z breve}. Pick a $K\subset \H_{V}(\A^{\infty})$ fixing $f, f', \phi, \phi'$, and
let $T\in \sH(\H(\A^{\infty}))$ be a Hecke operator acting as $\vol(K)^{-1}\rt(f\ot f'^{\vee})$ on  $\sg^{\vee}$. Then \eqref{my AIP} is equivalent to \cite[Theorem 1.8 (1)]{DL} in level $K$ for 
$$(\vphi, T\phi; \vphi', \phi').$$
   (Note that our definitions of the arithmetic theta lifts  $\Theta(-, -)$ differ from those of \cite{DL} by  a factor $\vol(K)$; in the height formula, one factor is accounted for by \eqref{vol here}, and another one by the normalisation of height pairings in \emph{loc. cit.}. The term $\vol^{\natural}(K)$  in \cite{DL} equals our $\vol(H_{\infty})\vol(K)$: this difference is accounted for 
by the factor $\vol(H_{\infty})$ in the pairing $\zeta_{\al}$.)
   
The $p$-adic  height vanishing formula is likewise equivalent to \cite[Theorem 1.8 (2)]{DL}.

The complex case is similarly reduced to \cite[Theorem 1.8]{LL2}. As $\rho$ is crystalline at all $w\vert p$, the representations $\Pi_{w}$, $\sg_{w}$, $\pi_{w}$ are unramified, so that we can take representatives $(f, \vphi, \phi; f', \vphi', \phi')$ of $\lm, \lm'\neq 0$ that are fixed by a maximal hyperspecial $K_{p}^{\circ}$. Then the fact that $\lan \ ,\ \ran^{\iota}$ is well-defined on Theta cycles follows from the definitions and \cite[Proposition 6.10 (3)]{LL}.
\end{proof}
Part 2 of Theorem \ref{main} is then an immediate consequence of Theorem \ref{ht f}.
For a beautiful exposition  of some key aspects of the proofs of the  formulas in \cite{LL, LL2, DL}, see \cite{Chao}. 

The proof of Theorem \ref{ht f} suggests that from the point of view of height formulas, Theta cycles offer no material  advantage over previous constructions. This is not so from the point of view of Euler systems, as we explain next.

\subsection{An Euler system} \lb{sec: euler}
The main technique for bounding Selmer groups is that of Euler systems, originally introduced by Kolyvagin to study Heegner points \cite{Kol, Koly}. Roughly speaking, an Euler system for a representation $\rho$ of $G_{E}$ is a collection of integral Selmer classes defined over certain abelian extensions of $\rho$ and satisfying certain compatibility relations; the (one) class defined over $E$ itself is called the \emph{base} class of the Euler system. 
	
In a forthcoming work, Jetchev--\nek--Skinner theorise a variant of this notion, that we shall call a \emph{JNS} Euler system.  It is adapted to conjugate-symplectic representations over CM fields, where the abelian extensions are ring class fields ramified at the primes of $E$ split over the totally real subfield~$F$ (see  \cite{Skinner}).
 Their main result is that if $\rho$ has `sufficiently large' image, then the existence of a JNS Euler system with nontrivial base class $z$ implies that $z$ generates the Selmer group of $\rho$: for a precise statement (when $F=\Q$), see  \cite[Theorem 8.3 and Remark 8.4]{ACR}, where JNS Euler systems are called `split anticyclotomic Euler systems' (\emph{ibid.}, Definition 8.1).

The following is the main result of  \cite{D-euler}.  Granted the results of Jetchev--\nek--Skinner, it implies part 3 of Theorem \ref{main}.

\begin{theo} \lb{ES} Let $\rho \colon G_{E}\to \GL_{n}(\Qpb)$ be a representation satisfying the assumptions of \S~\ref{sec:ass}. Then for any $\lm \in \Lm_{\rho}$, there exists a JNS Euler system  based on  $\Theta_{\rho}(\lm)$. 
\end{theo}

Multiplicity-one principles are remarkably useful to prove relations between special cycles and, in particular, compatibility relations in Selmer groups -- as first observed in  \cite{YZZ} and \cite{LSZ}. The proof of Theorem \ref{ES} is no exception: this is the main technical advantage of having constructed a cycle depending on one parameter only.


\begin{bibdiv}
\begin{biblist}

 

\bib{ACR}{article}{
   author={Alonso, Ra\'ul}, 
   author={Castella, Francesc}, 
   author={Rivero, \'Oscar}, title={The diagonal cycle Euler system for $\GL_{2}\times \GL_{2}$}, journal={Journal of the Institute of Mathematics of Jussieu}, publisher={Cambridge University Press},
 year={2023}, pages={1--63}}


\bib{Ato}{article}{
   author={Atobe, Hiraku},
   title={On the uniqueness of generic representations in an $L$-packet},
   journal={Int. Math. Res. Not. IMRN},
   date={2017},
   number={23},
   pages={7051--7068},
   issn={1073-7928},
   review={\MR{3801418}},
   doi={10.1093/imrn/rnw220},
}




\bib{bei}{article}{
   author={Be\u\i linson, A. A.},
   title={Higher regulators and values of $L$-functions},
   language={Russian},
   conference={
      title={Current problems in mathematics, Vol. 24},
   },
   book={
      series={Itogi Nauki i Tekhniki},
      publisher={Akad. Nauk SSSR, Vsesoyuz. Inst. Nauchn. i Tekhn. Inform.,
   Moscow},
   },
   date={1984},
   pages={181--238},
   review={\MR{760999}}, 
}



\bib{Bei87}{article}{
   author={Be\u{\i}linson, A.},
   title={Height pairing between algebraic cycles},
   conference={
      title={Current trends in arithmetical algebraic geometry},
      address={Arcata, Calif.},
      date={1985},
   },
   book={
      series={Contemp. Math.},
      volume={67},
      publisher={Amer. Math. Soc., Providence, RI},
   },
   date={1987},
   pages={1--24},
   review={\MR{902590}},
   doi={10.1090/conm/067/902590},
}



\bib{Ben}{article}{
   author={Benois, Denis},
   title={On extra zeros of $p$-adic $L$-functions: the crystalline case},
   conference={
      title={Iwasawa theory 2012},
   },
   book={
      series={Contrib. Math. Comput. Sci.},
      volume={7},
      publisher={Springer, Heidelberg},
   },
   date={2014},
   pages={65--133},
   review={\MR{3586811}},
}
	
	
	
\bib{BP-Planch}{article}{
   author={Beuzart-Plessis, Rapha\"{e}l},
   title={Plancherel formula for ${\rm GL}_n(F)\backslash {\rm GL}_n(E)$ and
   applications to the Ichino-Ikeda and formal degree conjectures for
   unitary groups},
   journal={Invent. Math.},
   volume={225},
   date={2021},
   number={1},
   pages={159--297},
   issn={0020-9910},
   review={\MR{4270666}},
   doi={10.1007/s00222-021-01032-6},
}


 \bib{BK}{article}{
   author={Bloch, Spencer},
   author={Kato, Kazuya},
   title={$L$-functions and Tamagawa numbers of motives},
   conference={
      title={The Grothendieck Festschrift, Vol.\ I},
   },
   book={
      series={Progr. Math.},
      volume={86},
      publisher={Birkh\"auser Boston, Boston, MA},
   },
   date={1990},
   pages={333--400},
   review={\MR{1086888 (92g:11063)}},
}



\bib{BST}{article}{
   author={Burungale, Ashay A.},
   author={Skinner, Christopher},
   author={Tian, Ye},
   title={The Birch and Swinnerton-Dyer conjecture: a brief survey},
   conference={
      title={Nine mathematical challenges---an elucidation},
   },
   book={
      series={Proc. Sympos. Pure Math.},
      volume={104},
      publisher={Amer. Math. Soc., Providence, RI},
   },
   year={2021},
   pages={11--29},
   review={\MR{4337415}}, 
}


\bib{carayol}{article}{
   author={Carayol, Henri},
   title={Sur la mauvaise r\'eduction des courbes de Shimura},
   language={French},
   journal={Compositio Math.},
   volume={59},
   date={1986},
   number={2},
   pages={151--230},
   issn={0010-437X},
   review={\MR{860139 (88a:11058)}},
}



\bib{Dplf}{article}{
   author={Disegni, Daniel},
   title={$p$-adic $L$-functions via local-global interpolation: the case of
   ${\rm GL}_2\times {\rm GU}(1)$},
   journal={Canad. J. Math.},
   volume={75},
   date={2023},
   number={3},
   pages={965--1017},
   issn={0008-414X},
   review={\MR{4586838}},
   doi={10.4153/S0008414X22000256},
}

\bib{D-euler}{article}{
   author={Disegni, Daniel},
   title={Euler systems for conjugate-symplectic motives}
, status={available at \url{https://disegni-daniel.perso.math.cnrs.fr/}},  label={Dis}}


%

\bib{DL}{article}{author={Disegni, Daniel}, author={Liu, Yifeng}, title={A $p$-adic arithmetic inner product formula}, journal ={Invent. math.},  volume={236} ,  date={2024},
   number={1},
   pages={219--371},}


\bib{DZ}{article}{author={Disegni, Daniel}, author={Zhang, Wei}, title={Gan--Gross--Prasad cycles and derivatives of $p$-adic $L$-functions}, status={preliminary version available at \url{https://disegni-daniel.perso.math.cnrs.fr/}}, label={DZ}, }



\bib{EHLS}{article}{
   author={Eischen, Ellen},
   author={Harris, Michael},
   author={Li, Jianshu},
   author={Skinner, Christopher},
   title={$p$-adic $L$-functions for unitary groups},
   journal={Forum Math. Pi},
   volume={8},
   date={2020},
   pages={e9, 160},
   review={\MR{4096618}},
   doi={10.1017/fmp.2020.4},
}


\bib{FM95}{article}{
   author={Fontaine, Jean-Marc},
   author={Mazur, Barry},
   title={Geometric Galois representations},
   conference={
      title={Elliptic curves, modular forms, \& Fermat's last theorem},
      address={Hong Kong},
      date={1993},
   },
   book={
      series={Ser. Number Theory},
      volume={I},
      publisher={Int. Press, Cambridge, MA},
   },
   isbn={1-57146-026-8},
   date={1995},
   pages={41--78},
   review={\MR{1363495}},
}


\bib{GS12}{article}{
   author={Gan, Wee Teck},
   author={Savin, Gordan},
   title={Representations of metaplectic groups I: epsilon dichotomy and
   local Langlands correspondence},
   journal={Compos. Math.},
   volume={148},
   date={2012},
   number={6},
   pages={1655--1694},
   issn={0010-437X},
   review={\MR{2999299}},
   doi={10.1112/S0010437X12000486},
}
	


\bib{GI14}{article}{
   author={Gan, Wee Teck},
   author={Ichino, Atsushi},
   title={Formal degrees and local theta correspondence},
   journal={Invent. Math.},
   volume={195},
   date={2014},
   number={3},
   pages={509--672},
   issn={0020-9910},
   review={\MR{3166215}},
   doi={10.1007/s00222-013-0460-5},
}
	
	
\bib{GI16}{article}{
   author={Gan, Wee Teck},
   author={Ichino, Atsushi},
   title={The Gross-Prasad conjecture and local theta correspondence},
   journal={Invent. Math.},
   volume={206},
   date={2016},
   number={3},
   pages={705--799},
   issn={0020-9910},
   review={\MR{3573972}},
   doi={10.1007/s00222-016-0662-8},
}
%



\bib{GRS}{book}{
   author={Ginzburg, David},
   author={Rallis, Stephen},
   author={Soudry, David},
   title={The descent map from automorphic representations of ${\rm GL}(n)$
   to classical groups},
   publisher={World Scientific Publishing Co. Pte. Ltd., Hackensack, NJ},
   date={2011},
   pages={x+339},
   isbn={978-981-4304-98-6},
   isbn={981-4304-98-0},
   review={\MR{2848523}},
   doi={10.1142/9789814304993},
}
	
\bib{GT}{article}{
   author={Gan, Wee Teck},
   author={Takeda, Shuichiro},
   title={A proof of the Howe duality conjecture},
   journal={J. Amer. Math. Soc.},
   volume={29},
   date={2016},
   number={2},
   pages={473--493},
   issn={0894-0347},
   review={\MR{3454380}},
   doi={10.1090/jams/839},
}



\bib{GG11}{article}{
   author={Gong, Z.},
   author={Greni\'{e}, L.},
   title={An inequality for local unitary theta correspondence},
   language={English, with English and French summaries},
   journal={Ann. Fac. Sci. Toulouse Math. (6)},
   volume={20},
   date={2011},
   number={1},
   pages={167--202},
   issn={0240-2963},
   review={\MR{2830396}},
}



	


\bib{gross-incoh}{article}{
   author={Gross, Benedict H.},
   title={Incoherent definite spaces and Shimura varieties}, 
   conference={title={Relative Trace Formulas}      },
   book={
      series={Simons Symposia},
publisher={Springer}, 
      place={Cham}}
      date={2021}
,pages={187--215},
   }


\bib{GZ}{article}{
   author={Gross, Benedict H.},
   author={Zagier, Don B.},
   title={Heegner points and derivatives of $L$-series},
   journal={Invent. Math.},
   volume={84},
   date={1986},
   number={2},
   pages={225--320},
   issn={0020-9910},
   review={\MR{833192 (87j:11057)}},
   doi={10.1007/BF01388809},  
}

\bib{GKZ}{article}{
   author={Gross, B.},
   author={Kohnen, W.},
   author={Zagier, D.},
   title={Heegner points and derivatives of $L$-series. II},
   journal={Math. Ann.},
   volume={278},
   date={1987},
   number={1-4},
   pages={497--562},
   issn={0025-5831},
   review={\MR{0909238}},
   doi={10.1007/BF01458081},
}

\bib{HKS}{article}{
   author={Harris, Michael},
   author={Kudla, Stephen S.},
   author={Sweet, William J.},
   title={Theta dichotomy for unitary groups},
   journal={J. Amer. Math. Soc.},
   volume={9},
   date={1996},
   number={4},
   pages={941--1004},
   issn={0894-0347},
   review={\MR{1327161}},
   doi={10.1090/S0894-0347-96-00198-1},
}




%




\bib{harris}{article}{
   author={Harris, Michael},
   title={Cohomological automorphic forms on unitary groups. II. Period
   relations and values of $L$-functions},
   conference={
      title={Harmonic analysis, group representations, automorphic forms and
      invariant theory},
   },
   book={
      series={Lect. Notes Ser. Inst. Math. Sci. Natl. Univ. Singap.},
      volume={12},
      publisher={World Sci. Publ., Hackensack, NJ},
   },
   date={2007},
   pages={89--149},
   review={\MR{2401812}},
}


\bib{Jo}{article}{
   author={Jochnowitz, Naomi},
   title={Congruences between systems of eigenvalues of modular forms},
   journal={Trans. Amer. Math. Soc.},
   volume={270},
   date={1982},
   number={1},
   pages={269--285},
   issn={0002-9947},
   review={\MR{642341}},
   doi={10.2307/1999772},
}
	
	
	

\bib{KMSW}{article}{author={Kaletha, Tasho}, author={Minguez, Alberto} , author={Shin,  Sug Woo Shin}, author={White,  Paul-James}
title ={Endoscopic classification of representations: Inner forms of unitary groups}, status={preprint}, label={KMSW}}


\bib{Kim}{article}{
   author={Kim, Chan-Ho},
   title={On the soft $p$-converse to a theorem of Gross-Zagier and
   Kolyvagin},
   journal={Math. Ann.},
   volume={387},
   date={2023},
   number={3-4},
   pages={1961--1968},
   issn={0025-5831},
   review={\MR{4657441}},
   doi={10.1007/s00208-022-02511-8},
}


\bib{KSZ}{article}{author={Kisin, Mark}, author={Shin,  Sug-Woo}, author={Zhu, Yihang}, title={The stable trace formula for Shimura varieties of abelian type}, status={arXiv:2110.05381}, label={KSZ}
}

\bib{Kol}{article}{
   author={Kolyvagin, V. A.},
   title={Finiteness of $E({\bf Q})$ and SH$(E,{\bf Q})$ for a subclass of
   Weil curves},
   language={Russian},
   journal={Izv. Akad. Nauk SSSR Ser. Mat.},
   volume={52},
   date={1988},
   number={3},
   pages={522--540, 670--671},
   issn={0373-2436},
   translation={
      journal={Math. USSR-Izv.},
      volume={32},
      date={1989},
      number={3},
      pages={523--541},
      issn={0025-5726},
   },
   review={\MR{954295 (89m:11056)}},
}



\bib{Koly}{article}{
   author={Kolyvagin, V. A.},
   title={Euler systems},
   conference={
      title={The Grothendieck Festschrift, Vol. II},
   },
   book={
      series={Progr. Math.},
      volume={87},
      publisher={Birkh\"{a}user Boston, Boston, MA},
   },
   date={1990},
   pages={435--483},
   review={\MR{1106906}},
}
	




\bib{Kud-Duke}{article}{
   author={Kudla, Stephen S.},
   title={Algebraic cycles on Shimura varieties of orthogonal type},
   journal={Duke Math. J.},
   volume={86},
   date={1997},
   number={1},
   pages={39--78},
   issn={0012-7094},
   review={\MR{1427845}},
   doi={10.1215/S0012-7094-97-08602-6},
}


\bib{Kud03}{article}{
   author={Kudla, Stephen S.},
   title={Modular forms and arithmetic geometry},
   conference={
      title={Current developments in mathematics, 2002},
   },
   book={
      publisher={Int. Press, Somerville, MA},
   },
   date={2003},
   pages={135--179},
   review={\MR{2062318}},
}

\bib{kudla-mod}{article}{
   author={Kudla, Stephen S.},
   title={Remarks on generating series for special cycles on orthogonal
   Shimura varieties},
   journal={Algebra Number Theory},
   volume={15},
   date={2021},
   number={10},
   pages={2403--2447},
   issn={1937-0652},
   review={\MR{4377855}},
   doi={10.2140/ant.2021.15.2403},
}

\bib{KRY06}{book}{
   author={Kudla, Stephen S.},
   author={Rapoport, Michael},
   author={Yang, Tonghai},
   title={Modular forms and special cycles on Shimura curves},
   series={Annals of Mathematics Studies},
   volume={161},
   publisher={Princeton University Press, Princeton, NJ},
   date={2006},
   pages={x+373},
   isbn={978-0-691-12551-0},
   isbn={0-691-12551-1},
   review={\MR{2220359}},
   doi={10.1515/9781400837168},
}





\bib{LR05}{article}{
   author={Lapid, Erez M.},
   author={Rallis, Stephen},
   title={On the local factors of representations of classical groups},
   conference={
      title={Automorphic representations, $L$-functions and applications:
      progress and prospects},
   },
   book={
      series={Ohio State Univ. Math. Res. Inst. Publ.},
      volume={11},
      publisher={de Gruyter, Berlin},
   },
   date={2005},
   pages={309--359},
   review={\MR{2192828}},
   doi={10.1515/9783110892703.309},
}




\bib{Chao}{article}{
   author={Li, Chao},
   title={From sum of two squares to arithmetic Siegel-Weil formulas},
   journal={Bull. Amer. Math. Soc. (N.S.)},
   volume={60},
   date={2023},
   number={3},
   pages={327--370},
   issn={0273-0979},
   review={\MR{4588043}},
   doi={10.1090/bull/1786},
}

\bib{LL}{article}{
   author={Li, Chao},
   author={Liu, Yifeng},
   title={Chow groups and $L$-derivatives of automorphic motives for unitary
   groups},
   journal={Ann. of Math. (2)},
   volume={194},
   date={2021},
   number={3},
   pages={817--901},
   issn={0003-486X},
   review={\MR{4334978}},
   doi={10.4007/annals.2021.194.3.6},
}



\bib{LL2}{article}{author={ Li, Chao}, author={Liu, Yifeng},
 title={Chow groups and L-derivatives of automorphic motives for unitary groups, II},
journal={Forum of Math. Pi}, volume={10}, date={2022}, pages={E5}}
%
%
%

\bib{Liu11}{article}{
   author={Liu, Yifeng},
   title={Arithmetic theta lifting and $L$-derivatives for unitary groups,
   I},
   journal={Algebra Number Theory},
   volume={5},
   date={2011},
   number={7},
   pages={849--921},
   issn={1937-0652},
   review={\MR{2928563}},
}

\bib{LiuII}{article}{
   author={Liu, Yifeng},
   title={Arithmetic theta lifting and $L$-derivatives for unitary groups,
   II},
   journal={Algebra Number Theory},
   volume={5},
   date={2011},
   number={7},
   pages={923--1000},
   issn={1937-0652},
   review={\MR{2928564}},
   doi={10.2140/ant.2011.5.923},
}



\bib{liu-fj}{article}{
   author={Liu, Yifeng},
   title={Fourier-Jacobi cycles and arithmetic relative trace formula (with
   an appendix by Chao Li and Yihang Zhu)},
   note={Appendix by Chao Li and Yihang Zhu},
   journal={Camb. J. Math.},
   volume={9},
   date={2021},
   number={1},
   pages={1--147},
   issn={2168-0930},
   review={\MR{4325259}},
   doi={10.4310/CJM.2021.v9.n1.a1},
}



\bib{LTXZZ}{article}{
   author={Liu, Yifeng},
   author={Tian, Yichao},
   author={Xiao, Liang},
   author={Zhang, Wei},
   author={Zhu, Xinwen},
   title={On the Beilinson-Bloch-Kato conjecture for Rankin-Selberg motives},
   journal={Invent. Math.},
   volume={228},
   date={2022},
   number={1},
   pages={107--375},
   issn={0020-9910},
   review={\MR{4392458}},
   doi={10.1007/s00222-021-01088-4},
}


\bib{LSZ}{article}{
   author={Loeffler, David},
   author={Skinner, Christopher},
   author={Zerbes, Sarah Livia},
   title={Euler systems for ${\rm GSp}(4)$},
   journal={J. Eur. Math. Soc. (JEMS)},
   volume={24},
   date={2022},
   number={2},
   pages={669--733},
   issn={1435-9855},
   review={\MR{4382481}},
   doi={10.4171/jems/1124},
}



\bib{McN}{book}{
   author={McNamara, Patrick},
   title={The neuroscience of sleep and dreams
},
series={Cambridge Fundamentals of Neuroscience in Psychology},
   publisher={Cambridge University Press, Cambridge,  UK},
   date={2019},
}

%
%

\bib{Min08}{article}{
   author={M\'{\i}nguez, Alberto},
   title={Correspondance de Howe explicite: paires duales de type II},
   language={French, with English and French summaries},
   journal={Ann. Sci. \'{E}c. Norm. Sup\'{e}r. (4)},
   volume={41},
   date={2008},
   number={5},
   pages={717--741},
   issn={0012-9593},
   review={\MR{2504432}},
   doi={10.24033/asens.2080},
}

 

\bib{Mok}{article}{
   author={Mok, Chung Pang},
   title={Endoscopic classification of representations of quasi-split
   unitary groups},
   journal={Mem. Amer. Math. Soc.},
   volume={235},
   date={2015},
   number={1108},
   pages={vi+248},
   issn={0065-9266},
   isbn={978-1-4704-1041-4},
   isbn={978-1-4704-2226-4},
   review={\MR{3338302}},
   doi={10.1090/memo/1108},
}


\bib{Mor}{article}{
   author={Morimoto, Kazuki},
   title={On the irreducibility of global descents for even unitary groups
   and its applications},
   journal={Trans. Amer. Math. Soc.},
   volume={370},
   date={2018},
   number={9},
   pages={6245--6295},
   issn={0002-9947},
   review={\MR{3814330}},
   doi={10.1090/tran/7119},
}



\bib{nek-CM}{article}{
   author={Nekov\'{a}\v{r}, Jan},
   title={The Euler system method for CM points on Shimura curves},
   conference={
      title={$L$-functions and Galois representations},
   },
   book={
      series={London Math. Soc. Lecture Note Ser.},
      volume={320},
      publisher={Cambridge Univ. Press, Cambridge},
   },
   date={2007},
   pages={471--547},
   review={\MR{2392363}},
   doi={10.1017/CBO9780511721267.014},
}
	

\bib{nek-height}{article}{
   author={Nekov{\'a}{\v{r}}, Jan},
   title={On $p$-adic height pairings},
   conference={
      title={S\'eminaire de Th\'eorie des Nombres, Paris, 1990--91},
   },
   book={
      series={Progr. Math.},
      volume={108},
      publisher={Birkh\"auser Boston},
      place={Boston, MA},
   },
   date={1993},
   pages={127--202},
   review={\MR{1263527 (95j:11050)}},
}







\bib{nek-AJ}{article}{
   author={Nekov\'a\v r, Jan},
   title={$p$-adic Abel-Jacobi maps and $p$-adic heights},
   conference={
      title={The arithmetic and geometry of algebraic cycles},
      address={Banff, AB},
      date={1998},
   },
   book={
      series={CRM Proc. Lecture Notes},
      volume={24},
      publisher={Amer. Math. Soc., Providence, RI},
   },
   date={2000},
   pages={367--379},
   review={\MR{1738867}},
}


\bib{nek-niz}{article}{
   author={Nekov\'a\v r, Jan},
   author={Nizio\l , Wies\l awa},
   title={Syntomic cohomology and $p$-adic regulators for varieties over
   $p$-adic fields},
   note={With appendices by Laurent Berger and Fr\'ed\'eric D\'eglise},
   journal={Algebra Number Theory},
   volume={10},
   date={2016},
   number={8},
   pages={1695--1790},
   issn={1937-0652},
   review={\MR{3556797}},
}


 
 


 

\bib{NT}{article}{
   author={Newton, James},
   author={Thorne, Jack},
   title={Symmetric power functoriality for Hilbert modular forms}, status={ arXiv:2212.03595}, label={NT}}
%
%




\bib{NY}{article}{
   author={Nishiyama, Kyo},
   author={Zhu, Chen-Bo},
   title={Theta lifting of holomorphic discrete series: the case of ${\rm
   U}(n,n)\times{\rm U}(p,q)$},
   journal={Trans. Amer. Math. Soc.},
   volume={353},
   date={2001},
   number={8},
   pages={3327--3345},
   issn={0002-9947},
   review={\MR{1828608}},
   doi={10.1090/S0002-9947-01-02830-6},
}

\bib{PT}{article}{
   author={Paul, Annegret},
   author={Trapa, Peter E.},
   title={One-dimensional representations of ${\rm U}(p,q)$ and the Howe
   correspondence},
   journal={J. Funct. Anal.},
   volume={195},
   date={2002},
   number={1},
   pages={129--166},
   issn={0022-1236},
   review={\MR{1934355}},
   doi={10.1006/jfan.2002.3974},
}

	

\bib{PR}{article}{
   author={Perrin-Riou, Bernadette},
   title={Points de Heegner et d\'eriv\'ees de fonctions $L$ $p$-adiques},
   language={French},
   journal={Invent. Math.},
   volume={89},
   date={1987},
   number={3},
   pages={455--510},
   issn={0020-9910},
   review={\MR{903381 (89d:11034)}},
   doi={10.1007/BF01388982},
}


\bib{PR-htIw}{article}{
   author={Perrin-Riou, Bernadette},
   title={Th\'{e}orie d'Iwasawa et hauteurs $p$-adiques},
   language={French},
   journal={Invent. Math.},
   volume={109},
   date={1992},
   number={1},
   pages={137--185},
   issn={0020-9910},
   review={\MR{1168369}},
   doi={10.1007/BF01232022},
}


 

\bib{PRbook}{book}{
   author={Perrin-Riou, Bernadette},
   title={Fonctions $L$ $p$-adiques des repr\'esentations $p$-adiques},
   language={French, with English and French summaries},
   journal={Ast\'erisque},
   number={229},
   date={1995},
   pages={198},
   issn={0303-1179},
   review={\MR{1327803 (96e:11062)}},
}
	
	
	\bib{Qiu}{article}{
   author={Qiu, Yannan},
   title={Generalized formal degree},
   journal={Int. Math. Res. Not. IMRN},
   date={2012},
   number={2},
   pages={239--298},
   issn={1073-7928},
   review={\MR{2876383}},
   doi={10.1093/imrn/rnr015},
}


\bib{Qiu2}{article}{
   author={Qiu, Yannan},
   title={Normalized local theta correspondence and the duality of inner
   product formulas},
   journal={Israel J. Math.},
   volume={191},
   date={2012},
   number={1},
   pages={227--278},
   issn={0021-2172},
   review={\MR{2970869}},
   doi={10.1007/s11856-011-0205-3},
}


\bib{Saito1}{article}{
   author={Saito, Takeshi},
   title={Modular forms and $p$-adic Hodge theory},
   journal={Invent. Math.},
   volume={129},
   date={1997},
   number={3},
   pages={607--620},
   issn={0020-9910},
   review={\MR{1465337}},
   doi={10.1007/s002220050175},
}
		
		
		
		\bib{Saito2}{article}{
   author={Saito, Takeshi},
   title={Weight-monodromy conjecture for $l$-adic representations
   associated to modular forms. A supplement to: ``Modular forms and
   $p$-adic Hodge theory'' [Invent. Math. {\bf 129} (1997), no. 3, 607--620;
   MR1465337 (98g:11060)]},
   conference={
      title={The arithmetic and geometry of algebraic cycles},
      address={Banff, AB},
      date={1998},
   },
   book={
      series={NATO Sci. Ser. C Math. Phys. Sci.},
      volume={548},
      publisher={Kluwer Acad. Publ., Dordrecht},
   },
   isbn={0-7923-6193-8},
   date={2000},
   pages={427--431},
   review={\MR{1744955}},
}
%
\bib{Sak}{article}{
   author={Sakellaridis, Yiannis},
   title={Plancherel decomposition of Howe duality and Euler factorization
   of automorphic functionals},
   conference={
      title={Representation theory, number theory, and invariant theory},
   },
   book={
      series={Progr. Math.},
      volume={323},
   },
   date={2017},
   pages={545--585},
   review={\MR{3753923}},
   doi={10.1007/978-3-319-59728-7\_18},
}

\bib{STay}{article}{
author={Sempliner, Jack}, author={Taylor, Richard},
title={On the formalism of Shimura varieties}, status={preprint}, label={ST}}


\bib{Skinner}{article}{
author={Skinner, Christopher},
title={Anticyclotomic Euler Systems}, status={Seminar at MSRI/SLMath, 30/03/203, recorded at \url{https://www.slmath.org/seminars/27455/schedules/33323} }, label={Ski}}

%

\bib{Tat65}{article}{
   author={Tate, John T.},
   title={Algebraic cycles and poles of zeta functions},
   conference={
      title={Arithmetical Algebraic Geometry},
      address={Proc. Conf. Purdue Univ.},
      date={1963},
   },
   book={
      publisher={Harper \& Row, New York},
   },
   date={1965},
   pages={93--110},
   review={\MR{0225778}},
}

%

    
\bib{tate-nt}{article}{author= {Tate, John},
     title = {Number theoretic background},
 conference = {title={Automorphic forms, representations and {$L$}-functions}, place= { {C}orvallis, {O}re.,},year={1977},}
       book={
      series={Proc. Sympos. Pure Math.}
      volume={XXXIII},
      publisher={Amer. Math. Soc.},
      place={Providence, R.I.},
   },
   year={1979}
     pages = {3--26},
}




\bib{Varma}{article}{
   author={Varma, Sandeep},
   title={On descent and the generic packet conjecture},
   journal={Forum Math.},
   volume={29},
   date={2017},
   number={1},
   pages={111--155},
   issn={0933-7741},
   review={\MR{3592596}},
   doi={10.1515/forum-2015-0113},
}
	

\bib{Yam14}{article}{
   author={Yamana, Shunsuke},
   title={L-functions and theta correspondence for classical groups},
   journal={Invent. Math.},
   volume={196},
   date={2014},
   number={3},
   pages={651--732},
   issn={0020-9910},
   review={\MR{3211043}},
   doi={10.1007/s00222-013-0476-x},
}



\bib{Wald-th}{article}{
   author={Waldspurger, J.-L.},
   title={D\'{e}monstration d'une conjecture de dualit\'{e} de Howe dans le cas
   $p$-adique, $p\neq 2$},
   language={French},
   conference={
      title={Festschrift in honor of I. I. Piatetski-Shapiro on the occasion
      of his sixtieth birthday, Part I},
      address={Ramat Aviv},
      date={1989},
   },
   book={
      series={Israel Math. Conf. Proc.},
      volume={2},
      publisher={Weizmann, Jerusalem},
   },
   date={1990},
   pages={267--324},
   review={\MR{1159105}},
}


\bib{Xia}{article}{
author={Xia, Jiacheng}, 
 title={Some cases of Kudla's modularity conjecture for unitary Shimura varieties}, volume={10}, DOI={10.1017/fms.2022.26}, journal={Forum of Mathematics, Sigma}, publisher={Cambridge University Press},  year={2022}, pages={e37}}

\bib{YZZ}{book}{
     title = {The Gross-Zagier Formula on Shimura Curves},  
     subtitle = {},     
     edition = {},       
     author = {Yuan, Xinyi},author = {Zhang, Shou-Wu},author = {Zhang, Wei},
     editor = {},     
     volume = {184},     
     series = {Annals of Mathematics Studies},  
     pages = {272},         
     place={Princeton, NJ},
     date = {2012},      
     publisher = {Princeton University Press},         
     }


\end{biblist}
\end{bibdiv}
\end{document}